\documentclass[reqno,11pt]{amsart}
\usepackage{a4wide,color,eucal,enumerate,mathrsfs}
\usepackage[normalem]{ulem}
\usepackage{amsmath,amssymb,epsfig,amsthm} 
\usepackage[latin1]{inputenc}
\usepackage{psfrag}

\def\rr{{\mathbb R}}

\def\fz{\infty}
\def\az{\alpha}

\def\dist{{\mathop\mathrm{\,dist\,}}}
\def\loc{{\mathop\mathrm{\,loc\,}}}
\def\lip{{\mathop\mathrm{\,Lip\,}}}

\def\ez{\epsilon}

\def\bint{{\ifinner\rlap{\bf\kern.25em--}
\int\else\rlap{\bf\kern.45em--}\int\fi}\ignorespaces}

\def\bbint{{\ifinner\rlap{\bf\kern.35em--}
\hspace{0.078cm}\int\else\rlap{\bf\kern.45em--}\int\fi}\ignorespaces}

\def\esup{\mathop\mathrm{\,esssup\,}}

\def\r{\right}

\newtheorem{thm}{Theorem}[section]
\newtheorem{lem}[thm]{Lemma}
\newtheorem{prop}[thm]{Proposition}
\newtheorem{cor}[thm]{Corollary}
\numberwithin{equation}{section}

\theoremstyle{remark}
\newtheorem{rem}[thm]{Remark}

\def\bint{{\ifinner\rlap{\bf\kern.35em--}
\int\else\rlap{\bf\kern.45em--}\int\fi}\ignorespaces}

\title{An asymtotic sharp Sobolev regularity  for planar infinity harmonic functions }
\author{Herbert Koch,  Yi Ru-Ya Zhang and Yuan Zhou}
\address{H. Koch: Institute of Mathematics, Bonn University, Endenicher Allee 60, Bonn 53115, Germany}
\email{koch@math.uni-bonn.de}
\address{Y. Zhang: Hausdorff Center for Mathematics, Endenicher Allee 62, Bonn 53115, Germany}
\email{yizhang@math.uni-bonn.de}
\address{Y. Zhou: Department of Mathematics, Beihang University, Beijing 100191, P.R. China}
\email{yuanzhou@buaa.edu.cn}
\date{\today}

\begin{document}

 \allowdisplaybreaks
\arraycolsep=1pt
\maketitle
\begin{center}
\begin{minipage}{13cm}
{\bf Abstract.}
Given an arbitrary planar $\infty$-harmonic function   $u$,
for each $\alpha>0$ we establish a quantitative $W^{1,2}_\loc$-estimate  of $|Du|^\alpha $,
which is  sharp as $\alpha\to0$. 
We also show that the distributional determinant of $u$ is  a Radon measure enjoying some quantitative  lower and upper bounds. 
As a by-product,  for  each $p>2$ we   obtain some quantitative $W^{1,p}_\loc$-estimates  of $u$,
and consequently,  an $L^p$-Liouville  property  for $\infty$-harmonic functions in whole plane.  

\end{minipage}
\end{center}

\section{Introduction}

Let $\Omega\subset \mathbb R^n$ be a domain (an open connected subset).   
 A  function $u\in C(\Omega)$  is  $\infty$-harmonic   in $\Omega$ if 
\begin{equation}\label{infty equ}
-\Delta_{\infty} u := -u_iu_ju_{ij}   =0 \quad {\rm in}\ \Omega
\end{equation}
in   viscosity sense; see \cite{j1993}.
In this paper,  
  $v_i=\frac{\partial v} {\partial x_i}$ if   $v\in C^1(\Omega)$, or $v_i$  denotes
 the distributional  derivation  in direction $i$  if   $v\in L^2_\loc(\Omega)$,  and
 $v_{ij}= \frac{\partial^2 v} {\partial x_i\partial x_j}$ if    $v\in C^2(\Omega)$. 
Write  $Dv=(v_i)_{i=1}^n$, $D^2v=(v_{ij})_{i,j=1}^n$, and $D^2vDv=(v_{ij}v_j)_{i=1}^n$.     
We always use the Einstein summation convention, that is, $f_ig_i=\sum_{i=1}^nf_ig_i$ for  vectors $(f_i)_{i=1}^n$ 
and $(g_i)_{i=1}^n$.

The main purpose is to prove the following quantitative  Sobolev regularity of  $\infty$-harmonic functions in planar    domains (that is, $n=2$). 

\begin{thm} \label{mainthm} Let $\Omega\subset\mathbb R^2$ be a domain and 
 $u$ be an $\infty$-harmonic function   in $\Omega$.
   For  each $\alpha>0$, we have  $|Du|^\alpha\in W^{1,2}_\loc(\Omega)$ with   
\begin{equation}\label{ds 1}
\| D|Du|^\alpha \|_{L^2(V)}\le  C (\alpha)\frac 1{\dist(V,\partial U)}
\||Du|^\alpha\|_{L^2(U)}\quad\forall V\Subset U\Subset \Omega,
\end{equation}
 and 
\begin{align}\label{point approx equ}
   (|Du|^\alpha)_i u_i=0  \quad {\rm almost\ everywhere\ in}\ \Omega. 
\end{align}
The constant  $C(\alpha)$ above depend  only on $\alpha$.
\end{thm}

As indicated by  $\infty$-harmonic function \begin{equation}\label{model}
w(x_1,x_2): =x_1^{4/3}- x_2^{4/3}\quad{\rm in}\ \mathbb R^2 \end{equation} 
given by the Aronsson \cite{a84}, we will see that $|Du|^\alpha\in W^{1,2}_\loc(\Omega)$  in Theorem \ref{mainthm}   is  sharp when   $\alpha\to0$. 
Precisely, set   \begin{equation}\label{palpha}p_\alpha:=3\quad{\rm if} \ \alpha\ge 1 \quad{\rm  and }\quad p_\alpha:=\frac{6}{3-\alpha} \quad{\rm if}\ \alpha\in(0,1).\end{equation}   
Note that $p_\alpha\to  2$ as $ \alpha\to0$.  By directly calculation,   we have the following result. 

\begin{lem} \label {aronssonf}
 For each $\alpha>0$, $$ |Dw|^\alpha\in W^{1,p}_\loc(\mathbb R^2)\quad \forall  p<p_\alpha \quad{\rm but} \quad |Dw|^\alpha\notin W^ {1,p_\alpha} _\loc(\mathbb R^2).$$  
Moreover,  
$$ \log |Dw| \in W^{1,p}_\loc(\mathbb R^2)\cap BMO_\loc(\mathbb R^2)\quad \forall  p<2 \quad{\rm but} \quad \log |Dw| \notin W^ {1,2} _\loc(\mathbb R^2).$$  
\end{lem}

 Below, we show that  the  distributional determinant  of any planar  $\infty$-harmonic function  is  a  Radon measure   enjoying some lower and upper bounds. 
 See Remark \ref{detrem} for the definition of distributional determinant for functions in $W^{1,2}_\loc(\Omega)$.

\begin{thm} \label{mainthm1}Let $\Omega\subset\mathbb R^2$ be a domain and 
 $u$ be an $\infty$-harmonic function   in $\Omega$.
Then  the  distributional determinant $-\det D^2u dx$ is a  Radon measure satisfying 
 $$-\det D^2u   \ge |D|Du||^2 $$ where $=$ holds when  $u\in C^2(\Omega)$,  
and 
 $$\|-\det D^2u\|(V)\le C \frac 1{[\dist(V,\partial U)]^2}
\||Du| \|^2_{L^2(U)}\quad\forall V\Subset U\Subset \Omega,$$
 where the constant $C$ above  is  absolute.
\end{thm}

Above $\|-\det D^2u\|(V)$ denotes the total measue of the open set $V$ with respect to the Radon measure $-\det D^2u\,dx$. 
Here, we list some remarks about  Theorem \ref{mainthm}, Lemma \ref{aronssonf} and Theorem \ref{mainthm1}.
\begin{rem}
(i) For planar $p$-harmonic function with $p>2$, recall that $|Du|^{(p-2)/2}Du \in W^{1,2}_\loc(\Omega)$  as proved in \cite{bi}.  
Theorem \ref{mainthm} can be viewed as some analogue for planar $\infty$-harmonic functions. 

(ii) Note that the function $u(x)=|x| \in W^{1,\infty}_\loc(\mathbb R^2)$  satisfies  $|Du|^2\in W^{1,2}_\loc(\mathbb R^2)$ and \eqref{point approx equ}, but is not $ \infty$-harmonic.  

(iii)  We make the following conjecture. 
\medskip 

\noindent {\it Conjecture.} Let $u $ be a planar $\infty$-harmonic function. Then the following hold:   
 
(a) $|Du|^\alpha\in W^{1,p}_\loc $ for   $2<p<p_\alpha$  and $\alpha>0$,   where $p_\alpha $ is given by \eqref{palpha}; 

(b) $-\det D^2u\in L^{p  }_\loc $ for  $1\le p<3/2$;  

(c) $ \log (|Du|^2+\kappa)\in BMO_\loc \cap W^{1,p}_\loc $ for  $p<2$ uniformly in $\kappa>0$. 
\end{rem}

The infinity Laplacian $ \Delta_{\infty} $ is a highly degenerate nonlinear second elliptic partial differential operator.  
The  equation \eqref{infty equ} is derived  by Aronsson 1960's as the Euler-Lagrange's equation when absolutely minimizing  
the $L^\infty$-functional $$\mathcal F_\infty(u,\Omega)=\esup_\Omega|Du|^2;$$ see \cite{a1,a2,a3,A1968,a4}.  
A function $u \in W^{1,\infty}_\loc(\Omega)$ is an absolute minimizer in $\Omega$ if 
$$\mathcal F_\infty(u, V)\le \mathcal F_\infty(v, V) $$
whenever $v\in W^{1,\infty}(V)$ and $v=u$ on $\partial V$.
Jensen  in 1993 identified the viscosity solutions of the equation \eqref{infty equ} (that is, $\infty$-harmonic functions)   with   absolute minimizers of such $L^\infty$-functional, see \cite{j1993}.

The $\infty$-harmonic functions (equivalently, absolute minimizers) are known to be  differentiable almost everywhere by \cite{j1993}; but not necessarily $C^2$ as shown by the Aronsson's function  in \eqref{model}.
The main issue in this direction to understand the possible regularity of 
$\infty$-harmonic functions.
Crandall et al.\ \cite{ceg} first  obtained the linear approximation property.  
Later, for  planar  $\infty$-harmonic functions, 
 the $C^1$-regularity  was  proved by Savin \cite{s05}, $C^{1,\alpha}$-regularity  with  $0<\alpha< 1$ by Evans-Savin \cite{es08}
and boundary $C^1$-regularity by \cite{wy}.  The key idea  is to establish a  flatness estimate by the planar topology and comparison property with cones, as first  observed by Savin \cite{s05}. 
When $n\ge 3$, the $C^1$ and $C^{1,\alpha}$-regularity of $\infty$-harmonic functions are still open. 
Recent progress is made by  Evans-Smart \cite{es11a,es11b}, who obtained the everywhere differentiability.
Their approach  is approximating the $\infty$-harmonic functions via  exponential harmonic functions (originally given by Evans  \cite{e03,ey04}),
and then establishing  a weaker flatness estimate via a PDE argument. 

Theorem \ref{mainthm} above gives  the asymptotic sharp Sobolev $W^{1,2}_\loc(\Omega)$-regularity of $ |Du|^\alpha$ for any  $\infty$-harmonic function $u$ in $\Omega \subset\mathbb R^2$ and $\alpha>0$. 
To prove  Theorem 1.1, we  also approximate $u$ via  exponential harmonic functions. 
 Precisely,  given an arbitrary domain $U\Subset \Omega$,  
for $\epsilon\in(0,1]$   let $u^\epsilon\in C^\infty(U)\cap C(\overline U)$ satisfy  
\begin{equation*}
-\Delta_{\infty} u^{\ez} -\ez \Delta u^{\ez}=0 \ \text{in $U$}, \quad u^{\ez}=u   \ \text{on $\partial U$}.
\end{equation*}
It is known that $u^\epsilon \to u$  uniformly in $\overline U$, see   \cite{ey04, es11b}  (or Theorem \ref{es2011} below). 
In this paper, we manage to show the  following  strong
$W^{1,p}_\loc(\Omega)$-convergence for  $1\le p<\infty$, 
which may  have its own interest. 

\begin{thm} \label{strong converge}
For each $\alpha>0$, we have $|Du^\epsilon|^\alpha\to |Du|^\alpha$  in $ L^p_{\loc}(U)$ for all $p\in[1,\infty)$ and weakly in $W^{1,2}_\loc(U)$. In particular,  $u^{\ez} \to u$ strongly in $W^{1,\,p}_{\loc}(U)$ for all $p\in[1,\infty)$. 
\end{thm}

 The proof of Theorem \ref{strong converge} relies on the uniform Sobolev  regularity in Lemma \ref{uniform 1}
and the  integral flatness estimate in Lemma  \ref{uniform 2} for approximating functions $u^\ez$. 
Observing  the following  identity 
\begin{equation}\label{identity}-\det D^2u^\ez=  |D| Du^\ez||^2 + \ez\frac{(\Delta u^\ez)^2}{|Du^\epsilon|^2} \quad {\rm almost\ everywhere\ in}\ U 
\end{equation}
in  Lemma \ref{I}, and by integration  against suitable test functions,  we obtain Lemma \ref{uniform 1} and Lemma \ref{uniform 2}, for the details see Section 5.  
The strong convergence in Theorem \ref{strong converge} permits us to conclude the Theorem \ref{mainthm} and Theorem \ref{mainthm1} 
from  the uniform Sobolev estimates in Lemma \ref{uniform 1}; see Section 3 for the proofs. 
The  asymptotic sharpness of Theorem \ref{mainthm} (that is, Lemma \ref{aronssonf}) will be also given after Theorem \ref{mainthm} in Section 3. 

Moreover, by  integration \eqref{identity} against some other test functions, we obtain some   quantitative Sobolev estimate of $u^\epsilon$ in Lemma \ref{sobolev of u}; for the details see also Section 5. 
The strong convergence in Theorem \ref{strong converge} allows to conclude  from  them the following  Theorem  \ref{mainthm0}, see Section 3 for the proof. 

\begin{thm} \label{mainthm0} Let $\Omega\subset\mathbb R^2$ be a domain and    $p>2$. 
For any  $\infty$-harmonic function  $u$ in $\Omega$  we have 
\begin{equation}\label{ds 2}
\| |Du| \|_{L^p(V)}\le  C(p) \frac 1{\dist(V,\partial U) }
\|  u-a \|_{L^p(U)}\quad\forall V\Subset U\Subset \Omega,\          \ \forall a\in\rr,
\end{equation}
where the constant $C(p) $ depends only on $p$. 
 \end{thm}

As a consequence of Theorem~\ref{mainthm} and Theorem \ref{mainthm0}, we obtain the following $L^p$-Liouville  property with $p>2$; see Section 3 for the proof.
Below, for $p\ge 1$  a function $u\in L^p_\loc(\mathbb R^2)$ satisfies the $L^p$-vanishing condition if 
\begin{align}\label{lpgrowth}\liminf_{R\to\infty} \frac { 1 }{R  }\left(\frac { 1 }{R^2 } \int_{ B(0,\,R)} | u(x)|^p \, dx\right)^{1/p}= 0.\end{align}

\begin{cor}\label{liuvioulle}
  If an $\infty$-harmonic function $u\in C(\mathbb R^2)$  satisfies  the $L^p$-vanishing condition for some $p>2$, 
  then $u$ must be a constant function. 
 \end{cor}

Note that the  $L^p$-vanishing condition in Corollary \ref{liuvioulle} is sharp in the sense that the planar $ \infty$-harmonic function $ u(x)=x_1$ satisfies 
$$\liminf_{R\to\infty} \frac { 1 }{R  }\left(\frac { 1 }{R^2 } \int_{ B(0,\,R)} | u(x)|^p \, dx\right)^{1/p}>0.$$
Recall that  if an $\infty$-harmonic function  $u\in C(\mathbb R^n)$ with   $n\ge2$  satisfying 
  $\lim_{x\to\infty}\frac{|u(x)|}{|x|}=0$ (that is, $L^\infty$-vanishing condition),  then $u$ must be a constant function  as proved by Crandall et al  \cite{ceg}.  
Savin \cite{s05} further proved that if an $\infty$-harmonic function  $u\in C(\mathbb R^2)$   satisfying 
  $\sup_{x\in \mathbb R^2}\frac{|u(x)|}{1+|x|}<\infty$, then $u$ must be a linear function; similar results  in higher dimension are still unknown. 

\begin{rem}
It would be interesting  to show that   Theorem \ref{mainthm0},  and hence Corollary \ref{liuvioulle}, holds  for  $ p\in[1,2]$.  
\end{rem}

Considering the duality relation between $p$-Laplacian and $q$-Laplacian when $1<p<\infty$ and $\frac1p+\frac1q=1$, 
an  interesting question  in this area is to understand  the possible dual relation between 1-Laplacian and $\infty$-Laplacian. 
We show that   $|Du|^2 $ satisfies a certain type of 1-Laplacian equation for $\infty$-harmonic function $u$ 
which is $C^2$ and does not have singular points; see Proposition~\ref{dual function}. 
For  Aronsson's function $w$, which is  singular on the axes, 
 there will be an extra term  appearing in the 1-Laplacian equation for $|Dw|^2$; see Lemma~\ref{dual example}.

Finally, we make some convention of notations.   We often write the 
constants as 
positive real numbers $C(\cdot)$ with
the parenthesis including all the parameters on which the constant depends; we just simply write $C$ if it is absolute and if there is no further explanation. The constant $C(\cdot)$ may
vary between appearances, even within a chain of inequalities. By $V\Subset U$ we mean that $V$ is a bounded domain of $U$
and $ \overline V\subset U$.

\section{Determinants of approximating functions}

 Suppose that $U\subset \mathbb R^2$ is a bounded domain.  We have the following general observation for the determinant of smooth functions. 
\begin{lem}\label{detlem}
For any smooth function $v$ in $U$ we have 
\begin{align}\label{det}
-\det D^2v 
&=-\frac12{\rm div}(\Delta v Dv-D^2vDv)=-\frac12(v_iv_j)_{ij}+\frac12(|Dv|^2)_{ii} \quad{\rm in}\ U,  
\end{align}
and 
\begin{align}\label{det2}
(-\det D^2v) |Dv|^2= |D^2v Dv|^2  -\Delta v\Delta_\infty v \quad{\rm in}\ U. 
\end{align}
\end{lem}
\begin{proof}
By direct  calculation we have   
\begin{align}-\det D^2v&=(v_{12})^2-v_{11}v_{22}\nonumber\\
&=\frac12[v_{ij}v_{ij}-v_{ii}v_{jj}]  = \frac 1 2 (|D^2v|^2 - (\Delta v)^2)\nonumber\\
&=\frac12 (v_{ij}v_j-v_{jj}v_i)_i =-\frac12{\rm div}(\Delta v Dv-D^2vDv)\nonumber\\
&=\frac12(v_jv_j)_{ii}-\frac12(v_iv_j)_{ij}  =-\frac12(v_iv_j)_{ij}+\frac12(|Dv|^2)_{ii}\nonumber,
\end{align} 
which gives \eqref{det}.
 
By direct  calculation we have 
\begin{align*}
 |D^2v Dv|^2 
&= (v_1v_{11}+v_2v_{12})^2+ (v_1v_ {21}+v_2v_{22})^2\\
&=v_{11} ((v_1)^2v_{11}+ 2v_1v_2 v_{12})+ v_{22} ((v_2)^2v_{22}+ 2v_1v_2 v_{12})+  (v_{12})^2((v_1)^2+(v_2)^2)\\
&=(v_{11}+v_{22} ) \Delta_\infty v-  v_{11}v_{22}((v_2)^2+(v_1)^2)+(v_{12})^2((v_1)^2+(v_2)^2)\\
&=\Delta v\Delta_\infty v +(-\det D^2v) |Dv|^2,
\end{align*}
 which gives \eqref{det2}. 
\end{proof}

\begin{rem}\label{detrem}
Given a function $v\in W^{1,2}_\loc(U)$, the distributional determinant $\det D^2v$ of $v$ is well defined as  given by 
$$\int_U-\det D^2v\phi\,dx=\frac12\int_U [-v_iv_j\phi_{ij}+|Dv|^2\phi_{ii}]\,dx\quad\forall \phi\in C_c^2(U).$$
\end{rem}

In the sequel of this section, for each $\epsilon >0$ we let $u^\epsilon\in C^\infty(U)$ be a solution to
\begin{equation}\label{infty ez}
-\Delta_{\infty} u^\epsilon -\ez \Delta u^ \epsilon=0 \quad \text{in $U$}.  
 \end{equation}
As a consequence of   Lemma \ref{detlem}, we have the following results. 
\begin{lem}\label{I} For each $\epsilon>0$ we have 
\begin{equation}\label{key inequalityII}(-\det D^2u^\ez) |Du^\ez|^2=  |D^2u^\ez Du^\ez|^2 + \ez(\Delta u^\ez)^2 \quad {\rm in}\ U.\end{equation}
Moreover, $ -\det D^2u^{\ez}\ge 0 \ {\rm in}\ U$ and 
 \begin{equation}\label{key inequalityIII} -\det D^2u^{\ez}= |D|Du^\ez||^2+\ez\frac{(\Delta u^\ez)^2 }{|Du^\ez|^2}  \quad {\rm almost\ everywhere\ in}\ U.\end{equation}
\end{lem}
\begin{rem}\label{remDelta0}
 If $Du^\ez(z)=0$ at $z\in U$, we then have  
$\ez\Delta u^\ez (z)=- \Delta_\infty u^\ez(z)=0$. 
 For this reason, we define 
$$ \frac{ \Delta u^\ez(z)  }{|Du^\ez(z)|^\beta}=0 $$  
for any $0<\beta<2$.
In particular, the last term in \eqref{key inequalityIII} is well defined in $U$. 
\end{rem}

\begin{proof} 
By  \eqref{det2} and $\Delta_\infty u^\ez=-\ez \Delta u^\ez$, we have 
$$  |D^2u^\ez Du^\ez|^2 =- \ez(\Delta u^\ez)^2+ (-\det D^2u^\ez) |Du^\ez|^2$$
 which gives  \eqref{key inequalityII}. 

Now we use \eqref{key inequalityII}  to show $-\det D^2u^{\ez}\ge 0$ in $U$. 
Let $\bar x\in U$ be an arbitrary fixed point.  If $|Du^\epsilon|(\bar x)\ne 0$, then by   \eqref{key inequalityII} we have $-\det D^2u^\ez(\bar x)\ge 0$. 
Assume that $Du^{\ez}(\bar x)=0$. If there exist  $\bar x^{(k)}$ such that $Du^{\ez}(\bar x^{(k)})\ne 0$ and $\bar x^{(k)}\to \bar x$ as $k\to\infty$, then by continuity, 
 $$-\det D^2u^{\ez}(\bar x)=\lim_{k\to\infty} -\det D^2u^\epsilon (\bar x^{(k)})\ge0.$$
 Otherwise, there exists  some sufficiently small $r>0$ such that $Du^{\ez}(  x)=0$ for all $x\in B(\bar x,r)$, and hence  $u$ is a constant function in  $B(\bar x,r)$, which implies that 
 $-\det D^2u^\epsilon (\bar x)=0$.

Finally,  \eqref{key inequalityII} also implies \eqref{key inequalityIII}. 
Indeed, notice that $|Du^\epsilon|\in \lip_\loc(U)$, hence $|Du^\epsilon|$ is differentiable almost everywhere.
Assume that $|Du^\epsilon|$ is differentiable at $z\in U$.
If $ Du^\ez (z) \ne 0$, then 
$$|D|Du^\epsilon|(z)|^2=\frac{|D^2u^\ez Du^\ez(z)|}{|Du^\ez|^2}= -\det D^2u^\ez-\ez\frac{(\Delta u^\ez)^2 }{|Du^\ez|^2}.$$
Assume now $ Du^\ez (z)= 0$. Considering Remark \ref{remDelta0}, 
we are required to show 
$$|D|Du^\ez|(z)|^2 =-\det D^2u^\ez(z).$$
Since $|Du^\epsilon|$ and $Du^\ez$  are differentiable at $z$, applying Taylor's expansion, we write 
$$|Du^\epsilon(x)|= \langle D |Du^\epsilon|(z),x-z\rangle +o(|x-z|)\quad\forall x$$
and $$ Du^\epsilon(x) =  D^2u^\epsilon (z)( x-z)  +o(|x-z|)\quad\forall x.$$
If  $D |Du^\epsilon|(z)\ne0$, 
pluging $x =z+t D|Du^\ez|(z)$ in both formula   and letting $t\to 0$, we obtain 
$$|D |Du^\epsilon|(z)|^2=  |D^2u^\epsilon (z)D|Du^\ez|(z)|.$$
Assuming $D|Du^\ez|(z)=|D|D  u^\ez|(z)| {\bf e}_1$ without loss of generality, we have 
$$|D|Du^\ez|(z)|^2=|D^2u^\ez {\bf e}_1|^2=(u^\ez_{11})^2+(u^\ez_{12})^2.$$
Since $\Delta u^\ez(z)=-\frac1\ez\Delta_\infty u^\ez(z)=0$, we obtain $u^\ez_{11} =-u^\ez_{22}$, which yields that 
$$|D|Du^\ez|(z)|^2 =-u^\ez_{11}u^\ez_{22} +(u^\ez_{12})^2=-\det D^2u^\ez(z).$$
If  $D |Du^\epsilon|(z)=0$, then $$ o(|x-z|)=|Du^\epsilon(x)|=|D^2u^\ez(z)(x-z)|+o(|x-z|)\quad\forall x.$$ 
Hence $D^2u^\ez(z)=0$, and 
$$|D|Du^\epsilon|(z)|^2=0  = -\det D^2u^\ez(z).$$
This completes the proof of Lemma \ref{I}. 
\end{proof}

Associated to such $u^\epsilon$, we introduce a functional $\mathbb I_\epsilon$ on  $ C_c (U)$ defined by 
\begin{align}
\mathbb I_\epsilon(\phi)   
&= \int_{U} -\det D^2u^\ez \phi\, dx\quad\forall \phi\in C_c(U) \label{identityII} 
\end{align} 
By  \eqref{key inequalityIII} we  write 
\begin{align}
\mathbb I_\epsilon(\phi)   
&= \int_{U}  |D|Du^\epsilon|(z)|^2 \phi \, dx +\ez\int_{U}  \frac{(\Delta u^\ez)^2}{|Du^\ez|^2} \phi \, dx\quad \forall  \phi\in  C_c (U). \label{identityIII} 
\end{align} 
By \eqref{det} and integration by parts,  we further write  
\begin{align}\label{functional}\mathbb I_\epsilon (\phi) &= \frac12\int_{U} [\Delta u^\ez u^\ez_i\phi_i  - u^\ez_{ij} u^\ez_j \phi_i]\,dx 
\quad \forall  \phi\in W^{1,\,2}_c (U)\\
&= \frac12  \int_{U}[-u^\ez_iu^\ez_j\phi_{ij}+ |Du^\epsilon|^2 \phi_{ii }] \,dx  \quad \forall  \phi\in C^2_c (U). \label{ds3}\end{align}

As a consequence of   \eqref{key inequalityIII} and   \eqref{ds3}, we have the following apriori estimates, which is uniform in $\epsilon>0$.  
\begin{cor}\label{I2}
We have \begin{align*}  
 \int_{V}  |D|Du^\epsilon|(z)|^2 \, dx +\ez\int_{V}  \frac{(\Delta u^\ez)^2}{|Du^\ez|^2}  \, dx&= \int_{V} -\det D^2u^\ez  \, dx\\
&\le \frac{8}{[\dist (V,\partial W)]^2}\int_W |Du^\ez|^2 \,dx\quad \forall  V\Subset W\Subset U .
\end{align*}
\end{cor}

Moreover,   by testing  $\phi= (|Du^{\ez}| ^2+\kappa)^{  \alpha-1  } \xi^4$  in   \eqref{functional}  for some suitable cut-off functions $\xi\in C_c^\infty(U)$ and $\kappa>0$,  
 applying \eqref{key inequalityIII} we have the following  estimates of $|Du^\epsilon|^\alpha$ for all $\alpha>0$, which are uniform in $\epsilon$. 
We postpone the details of the proof to Section 5. 

\begin{lem}\label{uniform 1}
 For $\alpha>0$,  $\epsilon\in(0,1]$ and $V\Subset W\Subset U$, we have 
\begin{align*}
 \int_V |D|Du^\epsilon|^\alpha|^2\,dx  +\epsilon   \int_V |Du^\epsilon|^{2\alpha-4} (\Delta u^\epsilon)^2\,dx   \le C (\alpha)\frac1{[\dist (V,\partial W)]^2} 
\int_W|Du^\epsilon| ^{2\alpha}\,dx    .
\end{align*}
\end{lem}

By testing  $\phi= (u^\epsilon - P)^2 \xi^4$   in   \eqref{functional}  for some suitable cut-off functions $\xi\in C_c^\infty(U)$ and any  linear function  $P$,  
 applying $-\det D^2u^\ez\ge 0$ in $U$ given by Lemma \ref{I} and Lemma \ref{uniform 1} we 
have the following uniform integral flatness.   The detailed proof  is   postpone to Section 5. 
 \begin{lem}\label{uniform 2}
For  any $ \bar x\in U$, $0<r<\dist(\bar x,\partial U)/4$ and linear function $P$, we have  
\begin{align*}
& \bint_{B(\bar x,r)} (|Du ^{\ez}|^2- \langle DP,Du^\epsilon\rangle  )^2 \,dx\\
&\quad \le C\left[\bint_{B(\bar x,2r)}|Du^\epsilon|^4\,dx\right]^{1/2} 
  \left[\bint_{B(\bar x,2r)} \left (\frac{|u^\epsilon-P |^2}{r^2}(|DP|+|Du^\epsilon|)^2   +    \frac{|u^\epsilon-P |^4}{r^4}\right)    \,dx \right]^{1/2} 
\end{align*}
\end{lem}

Finally, by testing $\phi= (|Du^{\ez}| ^2+\kappa)^{  \alpha-1} |u^\ez|^2  \xi^{2(\alpha+1)}$  in   \eqref{functional}  
for some suitable cut-off functions $\xi\in C_c^\infty(U)$ and $\kappa>0$, 
 applying \eqref{key inequalityIII}, we obtain the following estimates.  The detailed proof  is   postpone to Section 5. 

\begin{lem}\label{sobolev of u} 
For any $\alpha> 0$   and $\kappa >0$, we have
 \begin{align*} 
&  \int_{V} ( |Du^{\ez}|  ^2+\kappa)^{  \alpha+1 }   \, dx\\
&\quad\le  
  C (\alpha ) \frac1{[\dist(V,\partial W)]^{2(\alpha+1)}}  \int_{W}     |u^\ez|  ^{2 \alpha+2} \, dx+  (8\kappa+ \tilde C(\alpha)\epsilon)   
\int_W   (|Du^{\ez}| ^2+\kappa)^{\alpha }   \, dx.\\
&\quad+\tilde C( \alpha ) \epsilon       \frac1{ [\dist(V,\partial W)]^{2 }}  \int_W        (|Du^{\ez}| ^2+\kappa)^{\alpha-1 }  |u^\epsilon|^2   \, dx 
\end{align*}
for all $ V\Subset W\Subset U$. 
\end{lem}

\section{Proofs of main results}

Let $u \in C(\Omega)$ be an $\infty$-harmonic function in $ \Omega\subset\mathbb R^2$.
It is known that $u\in  W^{1,\infty}_\loc(\Omega)$  and  $u$ is differentiable almost everywhere,
that is, $Du$ exists almost everywhere.
Note that $Du$ also coincides with the weak derivative of $u$, and we abuse of the notation here for convenience. 

 By Evans \cite{e03} (see also \cite{ey04,es11b}), we know that, on subdomains of $\Omega$, 
 $u $ is  approximated by exponential harmonic functions.  To be precise, 
fix an arbitrary domain $U\Subset \Omega$.  
 For $\epsilon\in(0,1]$,  consider the following Dirichlet problem: 
\begin{equation}\label{regular infty equ}
\left\{ \begin{array}{rl}
-\Delta_{\infty} u^{\ez} -\ez \Delta u^{\ez}=0& \ \text{in $U$}\\
u^{\ez}=u &  \ \text{on $\partial U$}.
\end{array} \right.
\end{equation}
See \cite[Theorem 2.1]{es11b} for the following 
 properties. 
\begin{thm}\label{es2011}
For each $\epsilon\in(0,1]$, there exists a unique  solution $u^{\ez}\in C^\infty(U)\cap C(\overline U)$ to \eqref{regular infty equ}. 
 Moreover for all $\epsilon\in(0,1]$, we have
$$\max_{\overline{U}} |u^{\ez}|\le \max_{\partial U} u$$
and for every open set $V\Subset U$
$$\max_{ \overline{V}} |Du^{\ez}|\le C(\max_{\partial U} |u|, \dist(V,\,\partial U)).$$
where $C$ is independent of $\epsilon$. 
Furthermore, $u^{\ez}\to u$ uniformly on $\overline{U}$. 
\end{thm}

As a consequence of Lemma \ref{uniform 2} and Theorem \ref{es2011}, we have the following flatness. 
\begin{cor}\label{flatcor} Given $\epsilon\in(0,1]$,   $\bar x\in U$ and $r<\dist(\bar x,\partial U)/4$, if 
$$ \sup_{B(\bar x, 2r)}\frac{|u^\epsilon(x) - P(x) | }{r }\le  \lambda $$
for some linear function $P$ and $0< \lambda<1$, then 
\begin{align*}
& \bint_{B(\bar x,r)} (|Du ^{\ez}|^2- \langle DP,Du^\epsilon\rangle  )^2 \,dx\le 
 C(  \dist(\bar x,\partial U)) \lambda  
\end{align*}
\end{cor}

With the help of Theorem \ref{regular infty equ}, Lemma \ref{uniform 1} and Corollary \ref{flatcor}  (or Lemma \ref{uniform 2}), 
we are able to show the strong $L^p_\loc(\Omega)$-convergence of $Du^\epsilon$ for $1\le p<\infty$, that is,    Theorem~\ref{strong converge}. 

\begin{proof} [Proof of Theorem \ref{strong converge}.]
Fix $ \alpha>0$. By Lemma~\ref{uniform 1}  and  Theorem~\ref{es2011}, we know that $D|Du^\epsilon|^\alpha\in L^2_{\loc}(U)$ locally uniformly. 
From the weak compactness of $W_\loc^{1,2}(U)$, it follows that 
$ |Du^\epsilon|^\alpha $ converges, up to some subsequence,  to some function $f^{(\alpha)}$ in $L^p_\loc(U)$ and weakly  in $W^{1,2}_\loc(U)$. 
On the other hand, from  $u^\epsilon \to u$ in $C(\bar U)$ it follows that $Du^\epsilon$ converges to $Du$ weakly in $L^p (U)$.

We claim that $f^{(2)}=|Du|^2$ almost everywhere. Assume that the claim holds for the moment.  
Then for all $\az>0$, we have 
$$|Du^\epsilon|^\alpha=(|Du^\epsilon|^2)^{\alpha/2}\to (|Du|^2) ^{\frac \az 2}=|Du|^\az $$
almost everywhere as $\ez\to0$.
By Theorem~\ref{es2011} and Lebesgue's dominated convergence theorem, for any $p\in[1,\infty)$ and $V\Subset U$,
$$\lim_{\ez \to 0} \int_V ||Du^{\ez}|^\az-|Du|^\az|^p\, dx=0, $$
that is, $|Du^\epsilon|^\alpha\to |Du|^\alpha$ in $ L^p_{\loc}(U)$. 
When $\alpha=1$, this together with the weak convergence $Du^\epsilon \rightharpoonup Du$ in $L^p (U)$   shows that  
$Du^\epsilon\to Du$ strongly in $ L^p_{\loc}(U)$, that is,  $u^{\ez} \to u$ strongly in $W^{1,\,p}_{\loc}(U)$ for all $p\in[1,\infty)$. 

We prove the above claim below,  i.e.\ $f^{(2)}=|Du|^2$   almost everywhere. 
Assume that $u$ is differentiable at $\bar x$, and also assume that $\bar x$ is Lebesgue point of $f^{(2)}$ and $ Du $; the set of such $\bar x$ has full measure in $U$. 
Then for any $\lambda\in(0,1)$, there exists  $r_{\lambda,\bar x}\in (0, \dist(\bar x, \partial U)/8)$ such that for any $ r\in(0,r_{\lambda,\bar x})$, we have 
$$ \sup_{B(\bar x, 2r)}\frac{|u(x) - u(\bar x)-\langle Du(\bar x),(x-\bar x)\rangle | }{r }\le \lambda.$$

By Theorem~\ref{es2011}, for arbitrary     $ r\in(0,r_{\lambda,\bar x})$, there exists $\epsilon_{\lambda,\bar x,r}\in(0, 1]$ such that for all $\epsilon\in(0,\epsilon_{\lambda,\bar x,r })$, we have 
$$ \sup_{B(\bar x, 2r)}\frac{|u^\epsilon(x) -  u^\epsilon (\bar x)-\langle Du(\bar x),(x-\bar x)\rangle| }{r }\le 2\lambda.$$
Letting $P(x)= u^\epsilon (\bar x)-\langle Du(\bar x),(x-\bar x)\rangle$ in  Corollary \ref{flatcor}, we arrive at 
\begin{align*}
& \bint_{B(\bar x,r)} (|Du ^{\ez}|^2- \langle Du(\bar x),Du^\epsilon\rangle  )^2 \,dx\le 
 C(u,\dist(\bar x,\partial U)) \lambda  \quad\forall r\in(0,r_{\lambda,\bar x}),\  \epsilon\in(0,\epsilon_{\lambda,\bar x,r}).  
\end{align*}
On the other hand,  since 
$ |Du^\epsilon|^2 \to f^{(2)}$ in $L^2_\loc(U)$ and $Du^\epsilon \rightharpoonup Du$ weakly in $L^2_\loc(U)$, 
for any $ r\in(0,\dist(\bar x,\partial U/4) )$ we have 
\begin{align*}
 \bint_{B(\bar x,r)} (f^{(2)}- \langle Du(\bar x),Du \rangle  )^2 \,dx\le & \liminf_{\epsilon\to0}\bint_{B(\bar x,r)} (|Du ^{\ez}|^2- \langle Du(\bar x),Du^\epsilon\rangle  )^2 \,dx.
\end{align*}
Therefore,
\begin{align*}
 \bint_{B(\bar x,r)} (f^{(2)}- \langle Du(\bar x),Du \rangle  )^2 \,dx \le 
 C(u,\dist(\bar x,\partial U)) \lambda  \quad\forall r\in(0,r_{\lambda,\bar x}).
\end{align*}
Since $\bar x$ is a Lebesgue point of $f^{(2)}$ and $ Du $, via  H\"older's inequality, we obtain 
$$|f^{(2)}(\bar x)-|Du|^2( \bar x)|=\lim_{r\to 0} \bint_{B(0,r)} |f^{(2)}- \langle Du(\bar x),Du \rangle  |\,dx \le  C(u, \dist(\bar x,\partial U)) \lambda ^{1/2}. $$
By letting $\lambda\to0$, we have $f^{(2)}(\bar x)=|Du|^2(\bar x) $, and conclude the claim. 
\end{proof}

%
 
Now we are ready to prove Theorem~\ref{mainthm} and Theorem \ref{mainthm1} using Theorem  \ref{strong converge} and Lemma \ref{uniform 1}. 
\begin{proof}[Proof of Theorem 1.1]  Assume that $u\in C(\Omega)$ is $\infty$-harmonic in $\Omega\subset\mathbb R^2$.
Fix an arbitrary domain $U\Subset \Omega$ and let $u^\epsilon$ be the solution to the Dirichlet problem \eqref{regular infty equ} in $U$. 

Let us show the first part of the theorem. 
For $\alpha>0$,  by  Theorem~\ref{strong converge}, we know that $|Du^\epsilon|^\alpha$ weakly converges to  $ |Du|^\alpha$ in $W^{1,2}_\loc(U)$, 
and hence, together with Lemma~\ref{uniform 1}  and Theorem~\ref{es2011}, we obtain 
\begin{align*}
\|D|Du|^\alpha\|_{L^2(V)}&\le \liminf_{\epsilon\to0}\| D|Du^\epsilon|^\alpha\| _{L^2(V)}\\
&\le C(\alpha)\frac1{\dist (V,\partial W)}   \liminf_{\epsilon\to0}\||Du^\epsilon |^\alpha\| _{L^2(W)}\\
&    \le C(\alpha)\frac1{\dist (V,\partial W)}  \||Du |^\alpha\| _{L^2(W)}. 
\end{align*}

 Given $\alpha>0$, by the local strong convergence  $Du^\epsilon\to Du $ and the local weak convergence   $D|Du^\epsilon|^\alpha \rightharpoonup  D|Du |^\alpha$ in $L^2_\loc(U)$ shown in Theorem~\ref{strong converge}, 
 we have 
$$\int_U \langle D|Du|^\alpha,Du\rangle \phi\,dx=
\lim_{\epsilon\to0} \int_U \langle D|Du^\epsilon|^\alpha,Du^\epsilon\rangle \phi\,dx \quad\forall\phi\in C_c^\infty(U) 
$$
Note that $ D (|Du^\epsilon|^2+\kappa)^{\alpha/2}$ converges to $D |Du^\epsilon|^\alpha$ weakly in $L^2_\loc(U)$ as $\kappa\to0$,
we have 
\begin{align*}\int_U \langle D|Du|^\alpha,Du\rangle \phi\,dx&=
\lim_{\epsilon\to0}\lim_{\kappa\to0}  \int_U \langle  D (|Du^\epsilon|^2+\kappa)^{\alpha/2},Du^\epsilon\rangle \phi\,dx\\
& =
\lim_{\epsilon\to0}\lim_{\kappa\to0}  \int_U  \frac\alpha 2 ( |Du^\epsilon|^2+\kappa)^{\alpha/2-1} \langle  D  |Du^\epsilon|^2 ,Du^\epsilon\rangle \phi\,dx\\
& =
\lim_{\epsilon\to0}\lim_{\kappa\to0}  \int_U  \alpha   ( |Du^\epsilon|^2+\kappa)^{\alpha/2-1} \Delta_\infty u^\ez \phi\,dx   \quad\forall\phi\in C_c^\infty(U) 
\end{align*}
Applying $\Delta_\infty u^\ez =-\ez\Delta u^\ez$, we have 
\begin{align*}\int_U \langle D|Du|^\alpha,Du\rangle \phi\,dx& =
\lim_{\epsilon\to0}\lim_{\kappa\to0} -\ez \int_U   \alpha   ( |Du^\epsilon|^2+\kappa)^{\alpha/2-1} \Delta  u^\ez \phi\,dx \\
&= \lim_{\epsilon\to0} -
\alpha\ez \int_U |Du^\epsilon|^{\alpha-2}\Delta  u^\epsilon  \phi\,dx  \quad\forall\phi\in C_c^\infty(U) 
\end{align*}
 
Let $V= {\rm supp}\phi\Subset W\Subset U$. 
Notice that by  Lemma~\ref{uniform 1} and 
Theorem~\ref{es2011}, we obtain   
\begin{align*}
 \epsilon  \int_U|Du^\epsilon|^{2\alpha-4}(\Delta  u^\epsilon)^2  (\phi)^2\,dx
& \le C(  \phi )\epsilon  \int_V |Du^\epsilon|^{2\alpha-4}(\Delta  u^\epsilon)^2 \,dx\\
& \le C(\alpha, \phi,  \dist(V,\partial W)) \int_W |Du^\epsilon|^{2\alpha }  \,dx\\
& \le C(\alpha, \phi, u,\dist(V,\partial W),\dist(W,\partial U)) . 
\end{align*}
Thus via Young's inequality, we further have 
$$\left|-\alpha\epsilon \int_U |Du^\epsilon|^{\alpha-2}  \Delta  u^\epsilon  \phi\,dx\right| \le C\epsilon^{1/2} |V| + C
 \epsilon^{3/2} \int_U |Du^\epsilon|^{2\alpha-4} (\Delta  u^\epsilon)^2  (\phi)^2\,dx\to 0$$
as $\epsilon \to0$.  Therefore we conclude that 
$$\int_U \langle D|Du|^\alpha,Du\rangle \phi\,dx= 0\quad\forall\phi\in C_c (U) $$
as desired. 
\end{proof}

\begin{proof}[Proof of Lemma \ref{aronssonf}.]
A direct calculus gives that 
$$|Dw|  =\frac{4}3(x_1^{2/3}+ x_2^{2/3}) ^{1/2}\quad {\rm and}\quad
  |D|Dw|^2|   =\frac{4^2  }{3^3} (x_1^{-2/3}+ x_2^{-2/3} )^{1/2}  .$$
For any domain $U\Subset {\mathbb R}^2\setminus \{0\}$, since $|Dw|$  and $|D|Dw|^2|$ has upper and lower bounds on $\overline U$,
we know that  $|D|Dw|^\alpha|= \alpha 2^{-1}|Dw|^{ \alpha-2}  |D|Dw|^2| \in L^p(U)$ for all $p\ge 1$. 

Now assume that $0\in U\Subset {\mathbb R}^2$, and without loss of generality let $U=(-1,1)^2$. 
We have 
 \begin{align*}\int_{-1}^1\int_{-1}^1   \frac{|D|Dw|^2| ^p}{|Dw|^{p(2-\alpha)  }} \,dx&=
\frac{8\cdot 4^{2p}  }{3^{3p}} \int_{0 }^1\int_{0}^{x_1}\frac{ (x_1^{-2/3}+ x_2^{-2/3} )^{p/2}}{ (x_1^{2/3}+ x_2^{2/3}) ^{p(2-\alpha ) /2} } \,dx_2  dx_1\\
&=
\frac{8\cdot 4^{2p}  }{3^{3p}}\int_{0}^{1} \frac{ (1+t^{-2/3})^{p/2}  }{ (1+t^{ 2/3})^{p (2-\alpha)/2} 
  } \,dt  \int_{0 }^1     x_1^{-p/3}  
 x_1^{-p(2-\alpha ) /3}   x_1dx_1.
\end{align*}
The first integral is finite if and only if $p<3$, the second integral is finite if and only if $p<\frac{6}{3-\alpha}$ when $\alpha< 3$ and $p<\infty$ when $\alpha\ge3$. 

Therefore, we conclude that   $$|D|Dw|^\alpha|=  \alpha 2^{-1}|Dw|^{ \alpha-2  } |D|Dw|^2|  \in L^{ p}_\loc(\mathbb R^2)\quad \forall  p<p_\alpha \quad{\rm but} \quad  \notin L^ { p_\alpha} _\loc(\mathbb R^2) $$
and  
$$|D\log |Dw| |=   \frac12|Dw|^{  -2  } |D|Dw|^2|  \in L^{ p}_\loc(\mathbb R^2)\quad \forall  p<2 \quad{\rm but} \quad  \notin L^ { 2} _\loc(\mathbb R^2).$$
Moreover, $$\log |Dw| =\log\frac43+\frac12\log(x_1^{2/3}+x_2^{2/3}),$$
Since  $$ |x|^{2/3}\le x_1^{2/3}+x_2^{2/3} \le C|x|^{2/3},$$ by $\log |x|\in BMO(\mathbb R^2)$ we also have 
$ |D\log |Dw| | \in BMO_\loc(\mathbb R^2) $  as desired. 
\end{proof}
 

\begin{proof}[Proof of Theorem \ref{mainthm1}.]
 By Remark \ref{detrem}, $u\in  W^{1,\infty}_\loc(\Omega)$ allows to defined the distributional determinant $\det D^2v$, that is, 
$$\int_\Omega-\det D^2u\phi\,dx=\frac12\int_\Omega [-u_iu_j\phi_{ij}+|Du|^2\phi_{ii}]\,dx\quad\forall \phi\in C_c^2(\Omega).$$
Hence by Theorem \ref{strong converge}, 
$$\int_\Omega-\det D^2u\phi\,dx=\lim_{\ez\to0}\frac12\int_\Omega [-u^\ez_iu^\ez_j\phi_{ij}+|Du^\ez|^2\phi_{ii}]\,dx= \lim_{\ez\to0}\int_\Omega-\det D^2u^\ez \phi\,dx \quad\forall \phi\in C_c^2(\Omega).$$

Fix an arbitrary domain $U\Subset \Omega$ and let $u^\epsilon$ be the solution to the Dirichlet problem \eqref{regular infty equ} in $U$. 
A density argument shows that for all $\phi\in C_c (U)$,  we can define 
 $$\int_U-\det D^2u\phi\,dx=  \lim_{\ez\to0}\int_U-\det D^2u^\ez \phi\,dx.$$
Recalling  $-\det D^2u^\epsilon\ge 0$ in $U$ given by Lemma \ref{I}, 
we know that $-\det D^2u $ is indeed a nonnegative Radon measure. Moreover  the uniform upper estimates of $-\det D^2u^\epsilon $ yields that 
$$\|-\det D^2 u\|(V)\le\liminf_{\epsilon\to}\int_V-\det D^2u^\epsilon\,dx \le \frac{C}{[\dist(V, \partial W)]^2}\int_{W}|Du^\epsilon|^2 \,dx\quad\forall V\Subset W\Subset U.$$
Since $-\det D^2u^\epsilon \ge |D|Du^\epsilon||^2$ as given in Lemma \ref{I}, for $\phi\in C_c(U)$ with $\phi\ge 0$ by Theorem \ref{strong converge} we have 
 $$\int_U-\det D^2u\phi\,dx=\lim_{\ez\to0}\int_U-\det D^2u\phi\,dx   \ge \limsup_{\ez\to0}\int_U |D|Du^\epsilon||^2\phi\,dx\ge\int_U |D|Du ||^2\phi\,dx,$$
which yields that $-\det D^2u\,dx\ge |D|Du ||^2\,dx$ in $U$.  By the arbitrariness of $U\Subset\Omega$, we know that $-\det D^2u\,dx$ is a Radon measure enjoys the desired upper bounds and lower bounds. 

Finally, assume that $u\in C^2(\Omega)$.  If $Du (z)=0$ for some $z\in U$, then by  \cite{A1968} (see also  \cite{y06}),  
 $ u $ is a constant function in $\Omega$, and hence, we have $|D|Du| |^2=0=-\det D^2u  $ in $\Omega$. 
Now, we assume that $|Du|>0$ in $\Omega$. 
Up to  approximating $u$  in $C^2_{\loc}(\Omega)$ by smooth functions,
applying   Lemma \ref{detlem}, we  have  
$$(-\det D^2u ) |Du |^2=  |D^2u  Du |^2\quad{\rm in}\ \Omega.$$
Since $|Du|>0$ in $\Omega$, we have 
$$ -\det D^2u  =  \frac{|D^2u^\ez Du |^2}{|Du |^2}=|D|Du||^2\quad{\rm in}\ \Omega.$$
This completes the proof of Theorem \ref{mainthm1}.
\end{proof}

Using Lemma \ref{sobolev of u}, Theorem \ref{strong converge} and Theorem \ref{es2011}, we are able to prove Theorem \ref{mainthm0} as below. 
\begin{proof}[Proof of Theorem \ref{mainthm0}]
By Theorem \ref{es2011}, we know that the following term   
 $$ \int_W  (|Du^{\ez}| ^2+\kappa)^{\alpha }   \,dx   +    \frac1{[\dist(V,\partial W)]^{2 }}   \int_W  (|Du^{\ez}| ^2+\kappa)^{\alpha -1}  |u^\epsilon|^2   \, dx$$
appeared in Lemma \ref{sobolev of u} is bounded uniformly in $\epsilon $ for each fixed $\kappa>0$ when $\alpha\in(0,1)$ and for all $\kappa\in(0,1)$ when $\alpha\ge 1$. 
Letting $\epsilon \to0$ in  Lemma \ref{sobolev of u}, by Theorem \ref{strong converge} and Theorem \ref{es2011} we have 
 \begin{align*} 
  \int_{V} ( |Du |  ^2+\kappa)^{  \alpha+1 }   \, dx \le  
  & C (\alpha ) \frac1{[\dist(V,\partial W)]^{2(\alpha+1)}}  \int_{W}     |u |  ^{2 \alpha+2} \, dx +  8\kappa\int_W   (|Du | ^2+\kappa)^{\alpha }    \, dx.
\end{align*}
Letting $\kappa\to0$, we further obtain 
  $$\int_{V} |Du |  ^{ 2 \alpha+2 }   \, dx \le  C (\alpha ) \frac1{[\dist(V,\partial W)]^{2(\alpha+1)}}  \int_{W}     |u |  ^{2 \alpha+2} \, dx.$$
Observing $u-a$ is also $\infty$-harmonic, we know that the above also holds by replacing $u$ with $u-a$ for any $a\in\rr$. Thus Theorem \ref{mainthm0} holds with $p=2\alpha+2$. 
\end{proof}

Finally, we prove Corollary \ref{liuvioulle} using Theorem \ref{mainthm} and Theorem \ref{mainthm0}. 
\begin{proof}[Proof of Corollary \ref{liuvioulle}. ]  
Assume that $u\in C(\mathbb R^2)$ is $\infty$-harmonic satisfying $L^p$-vanishing condition.  By Theorem \ref{mainthm} with $\alpha=p/2$ and Theorem \ref{mainthm0},
we know that
 \begin{align*} 
\|D|Du|^\alpha\| _{L^2(\mathbb R^2)}&= \liminf_{R\to\infty}\|D|Du|^\alpha\|_{L^2( B(0,R))}\le C \liminf_{R\to\infty} \frac1{  R }\| |Du|^\alpha\| _{L^2( B(0,2R))}\\
&= C\liminf_{R\to\infty} \frac1{  R }\| |Du| \|^{p/2} _{L^p( B(0,2R))}\le C\liminf_{R\to\infty} \frac1{ R ^{1+2/p} } \|   u \|^{ p/2}_{L^p(  B(0,4R) )}.
\end{align*}
By $L^p$-vanishing condition,   we have $ \|D|Du|^2\|_{L^2(\mathbb R^2)}=0 $ and  hence $D|Du|^\alpha=0$ almost everywhere. 
Thus $|Du|=c$ almost everywhere.  By  Theorem \ref{mainthm0} again, we have 
$$ c^\alpha\le C\liminf_{R\to\infty}\frac1 R \| |Du|^\alpha\|_{L^2( B(0,2R))}= C \liminf_{R\to\infty}    \frac1{ R ^{1+2/p} } \|   u \|^{ p/2}_{L^p(  B(0,4R) )} . $$
 By $L^p$-vanishing condition again,   we have  $c=0$, that is, $u$ must be a constant function. 
\end{proof}

\section{The duality between the $1$-Laplacian and the $\infty$-Laplacian in the plane}
Let $U\subset  { \mathbb R^2}$ be a bounded domain. 
For a given pair of continua $E,\,F \subset \overline U$ and $1\le p\le \infty$, one defines the {\it $p$-capacity between $E$ and $F$ in $U$} as
 $${\rm Cap}_p(E,\,F;\,U)=\inf\{\|\nabla u\|^p_{ L^{p} (\Omega)}: \ u\in\Delta(E,\,F;\,U)\},$$
 where $\Delta(E,\,F;\,U)$ denotes the class of all $u\in W^{1,\,p}(\Omega)$ that are continuous
in $\Omega\cup E\cup F$ and satisfy $u=1$ on $E$, and $u=0$ on $F$. 
The following duality of capacities in the plane was established in \cite[pp.888-891]{R2008}, which   originally follows from \cite{Z1967}. 

\begin{lem}\label{pq duality}
Let $U\subset  {\mathbb R^2}$ be a Jordan domain enclosed by four arcs $\gamma_1$, $\gamma_2$, $\gamma_3$ and $\gamma_4$ counterclockwise. Then we have 
$$\left[{\rm Cap}_p(\gamma_1,\,\gamma_3;\,U)\right]^{\frac 1 p}\left[{\rm Cap}_q(\gamma_2,\,\gamma_4;\,U)\right]^{\frac 1 q}= 1$$
for $1\le p\le \infty$ and $q=\frac p {p-1}$. 
\end{lem}
When $1< p < \infty$ and $q=\frac p {p-1}$, 
this duality between capacities is related to the following equation system
$$
\left\{ \begin{array}{rl}
v_x=&|D u|^{p-2} u_y\\
v_y= &-|D u|^{p-2} u_y
\end{array} \right.
$$
in a domain $U\subset \mathbb R^2$; see \cite{R2008}. The function $u$ is $p$-harmonic and $v$ is $q$-harmonic, and their gradients are orthogonal to each other. This is a generalization of classical Cauchy-Riemann equations and was applied in  e.g.\ \cite{L1987}. Also it is related to the hodograph transformation, which, for example, was applied to show the sharp H\"older regularity of solutions to certain equations involving the $p$-Laplacian; see e.g.\ \cite{IM1989} and \cite{ATU2017}. We also refer to \cite[Chapter 16]{AIM2009} for more applications of hodograph transformation. 
Especially, our Lemma 2.3 is partially motivated by the lower estimate on the determinant of the Jacobian of the hodograph transformation (see e.g.\ \cite[Lemma 2.1]{BK2005}) . 

However, when $p=1$ or $\infty$ we no longer have such a nice equation system, even though the duality of capacities in Lemma~\ref{pq duality} still holds. The reason is that a $1$-harmonic  function  may not be continuous; it can even  be a summation of several characteristic functions of sets. Then it is not very meaningful to talk about the orthogonality of gradient between $1$-harmonic functions and $\infty$-harmonic functions. 

Nevertheless, when $u$ is a smooth infinity harmonic function, notice that $|Du|^2$ is constant along the gradient trajectory of $u$ and  $D|Du|^2$ is orthogonal to it. Then $|Du|^2$ behaves similar to a dual function of $u$ in the above sense. Indeed, motivated by \cite{e93,E2007} we have the following observation. 

\begin{prop}\label{dual function}
Let $U\subset\mathbb R^2$ be a domain. 
If $u \in  C^2(U)$ is an $\infty$-harmonic function so that $Du\neq 0$ and $\det D^2 u \neq 0$,
then $ v=\frac 1 2 |Du|^2$ satisfies the following equation
\begin{equation}\label{1-laplace}
-{\rm div}\left(\frac {Dv}{|D v|}\right)= \frac {|D v|}{2v}.
\end{equation}
The geometric meaning of the equation is that, the mean curvature of the level set of $v$, equivalently that of the gradient trajectory of $u$, is   $ {|D^2 uDu|}/{|Du|^2}$. 
\end{prop}
\begin{proof}
First of all, by \cite[Lemma 2]{A1968} we know that $u$ is smooth. 
Then a direct calculation via the equation of $u$ shows that
$$-\langle D^2v Du,\,Du\rangle=-(u_{ijk}u_i +u_{ij}u_{ik})u_j u_k = |D^2u Du|^2 =|Dv|^2. $$
Since we have assumed that $\det D^2 u \neq 0$, then $Dv\neq 0$. 
By the  orthogonality between $Du$  and $Dv$,  we then have
$$-2 v \langle D^2v\frac{Dv^\perp}{|Dv^\perp|},\,\frac{Dv^\perp}{|Dv^\perp|} \rangle=|Dv|^2.$$
In the plane we further deduce
$$- |Dv| {\rm div}\left(\frac {Dv}{|D v|}\right)= \frac {|Dv|^2}{2 v}.$$
As  $Dv\neq 0$, consequently we conclude the proposition.
\end{proof}

However in general \eqref{1-laplace} is not true; one can check that for $w=x_1^{4/3}- x_2^{4/3}$   in any neighborhood of the set where $D|D^2 w|=\infty$, i.e.\ the $x_1$-axis and $x_2$-axis.
Indeed,  there is another singular term on the right-hand side of \eqref{1-laplace}; see below.

\begin{lem}\label{dual example}
The function $v=\frac12|Dw|^{2}$ is a weak solution of the equation 
$$-{\rm div}\left(\frac {Dv}{|D v|}\right)= \frac {|D v|}{2v}- 2  (\mathcal H^1|_{\{x_1=0\}}+\mathcal H^1|_{\{x_2=0\}})\quad{\rm in}\ \mathbb R^2$$.
\end{lem}

\begin{proof}

For $\phi\in C_c^\infty(\mathbb R^2)$, write 
$$F(\phi)=  \int_{\mathbb R^2} \langle \frac {Dv}{|D v|}, D\phi\rangle \,dx - \frac {|D v|}{2v}\phi\,dx.$$
For $\eta\in\mathbb R$, write $$S_\eta=\{(x_1,\eta): x_1\in\mathbb R^1\}.$$

If  $\{(0,x_2)|x_2\in\mathbb R\}\cap   \rm spt \phi=\emptyset $, by the Green  identity and Proposition \ref{pq duality}, we have 
$$F(\phi)= -2\lim_{\eta\to 0_+}  \int_{S_\eta}  \frac { v_2}{|D v|} \phi  \,dx_1.$$
Note that $$v= \frac{4^2}{2\cdot 3^2}(x_1^{2/3}+ x_2^{2/3})  $$ and $$v_2=\frac{4^2}{  3^3}  x_2^{-1/3}\quad, |Dv|=\frac{4^2}{  3^3}  (x_1^{-2/3}+x_2^{-2/3})^{1/2} .$$
It follows that 
$$F(\phi)= -2\lim_{\eta\to 0_+}  \int_{x_1\in\mathbb R}  \frac { x_1^{-1/3}}{(x_1^{-2/3}+\eta^{-2/3})^{1/2}} \phi (x_1,\eta) \,dx_1= -2\int_{x_1\in\mathbb R}\phi (x_1,0) \,dx_1 $$
as desired. 

If  $\{(x_1,0)|x_1\in\mathbb R\}\cap   \rm spt \phi=\emptyset $,  we have similar result. 
If  $ (0,0)\in   \rm spt \phi=\emptyset $,  similarly, we have 
$$F(\phi)=   -2\int_{x_1\in\mathbb R}\phi (x_1,0) \,dx_1- 2\int_{x_2\in\mathbb R}\phi (0,x_2) \,dx_2 $$
as desired. 
\end{proof}

 \section{Proofs of the Lemmas \ref{uniform 1} to \ref{sobolev of u}}
Suppose that $U\subset \mathbb R^2$ is a bounded domain, and for $\epsilon\in(0,1)$, let $u^\epsilon\in C^\infty(U)$ be a solution to   the equation \eqref{infty ez}.

\begin{lem}\label{cacciopoli ez} 
For any $\alpha>0$
and  $\xi \in C^{\infty}_c(U)$, we have
\begin{align*}
& \int_{U}  |D |Du^\epsilon |^\alpha|^2 \xi^2 \,dx +
 {\ez}   \int_U|Du^{\ez}|^ {2\alpha-4}    |\Delta u^{\ez}|^2  \xi^2\,dx\le C (\alpha)\int_{U}|Du^{\ez}| ^{2\alpha }(|D\xi|^2 +|D^2\xi|\,|\xi|)\,dx ,
\end{align*}
where the constant $C$ is absolute.
\end{lem}

We obtain Lemma \ref{uniform 1} as immediate consequence by choosing $\xi\in C_c(U)$ so that $\xi=1$ on $U$ with 
$$|D\xi|\le  \frac{2}{ \dist(V,\partial U)} \quad{\rm and}\quad |D^2\xi|\le \frac{C}{[\dist(V,\partial U)]^2}.$$

\begin{proof}[Proof of Lemma \ref{cacciopoli ez}.]
The case $\alpha=1$ is already proved in Corollary \ref{I2} via taking  $\phi=\xi$. Now we assume that $\alpha\in (0,1)\cup(1,\infty)$.

Let $\phi= (|Du^{\ez}| ^2+\kappa)^{  \alpha-1 } \xi^2$  for $\kappa>0$ and $ \xi\in C_c^\infty(U)$. Then $\phi\in W^{1,\,2}_c(U)$. 
By \eqref{identityIII}, we write 
\begin{align*}\mathbb I_\epsilon(\phi)&= \int_{U}  |D|Du^\epsilon|(z)|^2 \phi \, dx +\ez\int_{U}  \frac{(\Delta u^\ez)^2}{|Du^\ez|^2} \phi \, dx\\
&=\int_{U}  |D|Du^\epsilon|(z)|^2 (|Du^{\ez}| ^2+\kappa)^{  \alpha-1  } \xi^2 \, dx +\ez\int_{U}  \frac{(\Delta u^\ez)^2}{|Du^\ez|^2} (|Du^{\ez}| ^2+\kappa)^{  \alpha-1 } \xi^2\, dx.
\end{align*}
Since 
$$ |D|Du^\epsilon|(z)|^2  \ge  |D^2 u^{\ez} Du^{\ez}|^2 (|Du^{\ez}| ^2+\kappa)^{-1} \quad {\rm and } \quad\frac{(\Delta u^\ez)^2}{|Du^\ez|^2}\ge (\Delta u^\ez)^2(|Du^{\ez}| ^2+\kappa)^{-1}  $$
almost everywhere we obtain
\begin{equation}
\label{inequ} 
\mathbb I_\epsilon(\phi)  \ge \int_{U}  |D^2u^\epsilon Du^\epsilon|^2 (|Du^{\ez}| ^2+\kappa)^{  \alpha-2  } \xi^2 \, dx +\ez\int_{U}  (\Delta u^\ez)^2 (|Du^{\ez}| ^2+\kappa)^{  \alpha-2 } \xi^2\, dx.
\end{equation}
 On the other hand, note that $$  \phi_i=  2(\alpha-1)(|Du^{\ez}| ^2+\kappa)^{\alpha-2}  u^{\ez}_{ik}u^{\ez}_k\xi^2+ 2\xi \xi_i (|Du^{\ez}| ^2+\kappa)^{\alpha-1} .$$
Pluging  again  $ \phi$ in \eqref{functional}, we obtain a second expression for $I_\epsilon$,  
\begin{align*}
 \mathbb I_\epsilon(\phi)&= \frac12\int_{U} [\Delta u^\ez u^\ez_i\phi_i  - u^\ez_{ij} u^\ez_j \phi_i]\,dx \\
& = - (\alpha-1)\int_{U}  (|Du^{\ez}| ^2+\kappa)^{\alpha-2}|D^2 u^{\ez} Du^{\ez}|^2\xi^2\,dx\\
 &\quad-\int_{U}(|Du^{\ez}| ^2+\kappa)^{\alpha-1} u^{\ez}_{ij}u^{\ez}_j\xi_i \xi\\
&\quad +(\alpha-1) \int_{U}(|Du^{\ez}| ^2+\kappa)^{\alpha-2} \Delta u^{\ez} u^{\ez}_{ik}u^{\ez}_i u^{\ez}_k \xi^2 \,dx\\
&\quad+  \int_{U}(|Du^{\ez}| ^2+\kappa)^{\alpha-1} \Delta u^{\ez} u^{\ez}_i\xi_i   \xi \, dx.
\end{align*}

Replacing $\Delta_\fz u^\ez=u^{\ez}_{ik}u^{\ez}_ku^\ez_i$ by $-\epsilon\Delta u^\ez$ in the third term - which we may since $u^\varepsilon$ satisfies \eqref{infty ez}, we further have
$$(\alpha-1) \int_{U}(|Du^{\ez}| ^2+\kappa)^{\alpha-2} \Delta u^{\ez} u^{\ez}_{ik}u^{\ez}_i u^{\ez}_k \xi^2 \,dx= 
-(\alpha-1)\ez \int_{U}(|Du^{\ez}| ^2+\kappa)^{\alpha-2} (\Delta u^{\ez})^2   \xi^2 \,dx.$$
Taking into account \eqref{inequ}  we conclude that

\begin{align}
& \alpha \int_{U}  (|Du^{\ez}| ^2+\kappa)^{\alpha-2}|D^2 u^{\ez} Du^{\ez}|^2\xi^2\,dx +\alpha \ez\int_{U}  \frac{(\Delta u^\ez)^2}{|Du^\ez|^2} (|Du^{\ez}| ^2+\kappa)^{  \alpha-1 } \xi^2\, dx \nonumber\\
& \quad\le  -\int_{U}  u^{\ez}_{ij}u^{\ez}_j\xi_i(|Du^{\ez}| ^2+\kappa)^{\alpha-1} \xi^2\,dx +\int_U   \Delta u^{\ez} u^{\ez}_i\xi_i(|Du^{\ez}| ^2+\kappa)^{\alpha-1} \xi  \, dx. \label{equ 1}
\end{align}
For the second term of the right hand side of \eqref{equ 1}, via integration by parts we have
\begin{align}
   \int_{U}  \Delta u^{\ez} u^{\ez}_i\xi_i (|Du^{\ez}| ^2+\kappa)^{\alpha-1} \xi \, dx  
=-  \int_{U}  u^{\ez}_k (u^{\ez}_i\xi_i (|Du^{\ez}| ^2+\kappa) ^{\alpha-1} \xi)_k \, dx\hspace{-10cm} &  \nonumber\\
&=  -\int_{U}   u^{\ez}_{ik} u^{\ez}_k \xi_i (|Du^{\ez}| ^2+\kappa) ^{\alpha-1}\xi\,dx-\int_{U} (u_i^{\ez} \xi_i)^2(|Du^{\ez}| ^2+\kappa)^{\alpha-1}\,dx\nonumber\\
&\quad  - \int_{U}\xi_{ik} u^{\ez}_k u^{\ez}_i (|Du^{\ez}| ^2+\kappa)^{\alpha-1}  \xi\,dx-  (\alpha-1)\int_U u^{\ez}_i \xi_i u^{\ez}_{jk} u^{\ez}_j u^{\ez}_k (|Du^{\ez}| ^2+\kappa)^{\alpha-2} \xi  \, dx. \label{equ x2}
\end{align}
 For the sum of the two first term of the right hand side of \eqref{equ 1} and \eqref{equ x2}, via integration by parts we have 
\begin{align}
-2\int_{U}   u^{\ez}_{ij}u^{\ez}_j\xi_i(|Du^{\ez}| ^2+\kappa)^{\alpha-1} \xi \, dx
& =  -\frac2{\alpha }\int_{U}   [(|Du^{\ez}| ^2+\kappa)^{\alpha }]_i \xi_i\xi \, dx\nonumber\\
&=  \frac2{\alpha }\int_{U}    (|Du^{\ez}| ^2+\kappa)^{\alpha }  (|D \xi|^2 +\Delta\xi\, \xi) \, dx\nonumber\\
 &\le  
\frac2{\alpha } \int_{U}  (|Du^{\ez}| ^2+\kappa)^{\alpha }(|D\xi|^2+ |D^2 \xi|\, |\xi|)\, dx.\label{equ x1}\end{align}
Observing the fact that $ -\xi_{ik} u^{\ez}_k u^{\ez}_i \le |Du^\ez|^2|D^2\xi|$, for the third term  in the right hand side of \eqref{equ x2}  we have 
$$ - \int_{U}\xi_{ik} u^{\ez}_k u^{\ez}_i (|Du^{\ez}| ^2+\kappa)^{\alpha-1}  \xi\,dx\le  
C\int_{U}  (|Du^{\ez}| ^2+\kappa)^{\alpha } |D^2\xi| |\xi| \,dx.$$
Noting 
$$u^{\ez}_i \xi_iu^{\ez}_{jk} u^{\ez}_j u^{\ez}_k\le |D^2u^\ez Du^\ez||Du^\ez|^2|D\xi|$$
 and applying Young's inequality,  for the forth term  in the right hand side of \eqref{equ x2} we  obtain  
\begin{align*}
&- (\alpha-1) \int_U u^{\ez}_i \xi_i u^{\ez}_{jk} u^{\ez}_j u^{\ez}_k (|Du^{\ez}| ^2+\kappa)^{\alpha-2} \xi^3  \, dx\\
&\quad\le \eta \int_U   |D^2u^\ez Du^\ez|^2 (|Du^{\ez}| ^2+\kappa)^{\alpha-2} \xi^2  \, dx+ 
  \frac{\alpha-1}{4\eta} \int_U    (|Du^{\ez}| ^2+\kappa)^{\alpha} |D\xi|^2  \, dx.
\end{align*}
for $\eta>0$. We collect all the estimates starting from \eqref{equ x2} 
\begin{align}
  &\int_{U} \Delta u^{\ez} u^{\ez}_i\xi_i (|Du^{\ez}|^2+\kappa)^{\alpha-1} \xi \, dx \nonumber  \\
&\quad
 \le\eta \int_U   |D^2u^\ez Du^\ez|^2 (|Du^{\ez}| ^2+\kappa)^{\alpha-2} \xi^2  \, dx\nonumber\\
&\quad\quad  +C(\eta,\alpha) \int_{U}  (|Du^{\ez}| ^2+\kappa)^{\alpha }(|D\xi|^2+  |D^2 \xi|\xi)\, dx\nonumber 
\end{align}
We use the estimate with  $\eta=\alpha/2$ and  arrive at 
\begin{align}
&  \int_{U}  (|Du^{\ez}| ^2+\kappa)^{\alpha-2}|D^2 u^{\ez} Du^{\ez}|^2\xi^4\,dx    + \ez \int_{U}(|Du^{\ez}| ^2+\kappa)^{\alpha-2}( \Delta u^{\ez} )^2  \xi^4 \,dx\nonumber\\
& \quad\le   C(\alpha) \int_{U}  (|Du^{\ez}| ^2+\kappa)^{\alpha }(|D\xi|^2+ |D^2 \xi|\xi)\xi^2\, dx\nonumber
\end{align}
 If  $\alpha\ge  2$  we conclude, by letting $\kappa \to 0 $, that 
\begin{align*}&  \int_{U}   |D^2u^{\ez} Du^{\ez}|^2  |Du^{\ez}|^  { 2\alpha-4} \xi^4    \,dx + \ez \int_{U} (\Delta u^{\ez})^2  |Du^{\ez}|^  {2\alpha-4}  \xi^4   \,dx\\
&\quad\le   C(\alpha)\int_{U}|Du^{\ez}| ^{2\alpha } \xi^2(|D\xi|^2 +|D^2\xi|\xi)\,dx. 
\end{align*}
Since $$|D|Du^\epsilon| ^{\alpha  }| = \alpha   |Du^\epsilon| ^{ \alpha-2}|D^2u^\epsilon Du^\epsilon| ,$$ 
we obtain the desired result. 

If  $\alpha< 2$ then $(\Delta u^{\ez}(z))^2  |Du^{\ez}(z)|^{2\alpha-4}$ is well-defined by Remark \ref{remDelta0}.  Since   $Du^{\ez}(z)=0$  implies $\Delta u^\ez (z)=0$, we  have  
$$(\Delta u^{\ez} )^2  |Du^{\ez} |^  {2\alpha-4}=\lim_{\kappa\to0}(|Du^{\ez}| ^2+\kappa)^{\alpha-2}( \Delta u^{\ez} )^2 $$
almost everywhere in $U$. 
Moreover, note that 
$$ (|Du^{\ez}| ^2+\kappa)^{\alpha-2}|D^2 u^{\ez} Du^{\ez}|^2=\frac1{\alpha^2} |D  (|Du^{\ez}| ^2+\kappa)^{\alpha/2}) |^2.$$
 Choosing suitable functions $\xi$ we deduce that  
$(|Du^{\ez}| ^2+\kappa)^{\alpha/2} \in W^{1,2}_\loc(U)$ with a uniform bound  for  $\kappa\in(0,1)$ on compact sets. 
Since $(|Du^{\ez}| ^2+\kappa)^\alpha\to  |Du^{\ez}|   ^{\alpha }$ almost everywhere as $\kappa\to 0$, we deduce  that $|Du^{\ez}|^{\alpha }\in W^{1,2}_\loc(U)$ and this convergence is indeed weakly in $W^{1,2}_\loc(U)$.
Therefore we conclude 
\begin{align*}
& \int_{U}  |D|Du^\epsilon| ^{\alpha }|^2\xi^2 \,dx +
 {\ez} \int_U|Du^{\ez}|^  {2\alpha-4}  |\Delta u^{\ez}|^2 \xi^2\,dx\\
&\quad\le \lim_{\kappa\to0}\left\{ \int_{U}  |D(|Du^\epsilon| ^2+\kappa)^{\alpha/2 }|^2\xi^2 \,dx +
 {\ez}   \int_U(|Du^{\ez}|^2+\kappa) ^ { \alpha-2}  |\Delta u^{\ez}|^2 \xi^2\,dx\right\}\\
&\quad\le C (\alpha)\int_{U}|Du^{\ez}| ^{2\alpha }(|D\xi|^2 +|D^2\xi|\, |\xi|)\,dx 
\end{align*}
as desired.  
\end{proof}

We then show  the following uniform flatness estimate. 
\begin{lem}\label{flat}
For any $\xi \in C^{\infty}_c(U)$ and linear function $P$, we have
\begin{align*}
&\int_{U} (|Du ^{\ez}|^2- \langle DP,Du^\epsilon\rangle  )^2\xi^2 \,dx\\
&\quad\le \left[
\int_{U} |Du^{\ez}|^4(|D\xi|^2 +|D^2\xi||\xi|)\,dx\right]^{1/2}\left[  \int_U |u^\epsilon-P    |^2(|DP|^2+|Du^\epsilon | ^2)\xi^2\,dx \r.\\
&\quad\quad+
 \left. \int_U  |u^\epsilon-P|^4(|D\xi |^2+|D^2\xi ||\xi|) \,dx \right]^{1/2} , 
\end{align*}
where the constant $C$ is absolute.
\end{lem}

Lemma~\ref{uniform 2}   follows from Lemma \ref{flat}  via suitable 
choice of $\xi$.

\begin{proof}[Proof of Lemma \ref{flat}.] Without loss of generality, we may assume that
$P(x)= cx_2$. Then  $ |c|= |DP|$, $DP=c{\bf e}_2$ and $\langle Du^\epsilon  , DP\rangle= cu^\epsilon _2$. Let $\phi= (u^\epsilon-cx_2)^2  \xi^2\in W_c^{1,2}(U)$. Since $-\det D^2u^\ez\ge 0$ by Lemma \ref{I}, we have 
 \begin{align*}\mathbb I_\epsilon(\phi)&= \int_{U}   (-\det D^2u^\epsilon) \phi \, dx\ge 0.
  \end{align*}
Now 
$$ \phi_i =2 (u^\epsilon_i-c\delta_{2i})  \xi^2+2(u- cx_2)^2\xi\xi _i. $$
By  \eqref{functional},  we obtain 
\begin{align}
 \mathbb I_\epsilon (\phi)
&= \frac12\int_{U} [\Delta u^\ez u^\ez_i\phi_i  - u^\ez_{ij} u^\ez_j \phi_i]\,dx \\
&=  \int_{U} \Delta u^\ez u^\ez_i  (u^\epsilon_i-c\delta_{2i})  \xi^2 (u^\epsilon-cx_2) \,dx +  \int_{U}\Delta u^\ez u^\ez_i \xi_i\xi (u^\epsilon-cx_2)^2\,dx \nonumber
    \\
& \quad - \int_U u^\ez_{ij} u^\ez_j (u^\epsilon_i-c\delta_{2i})  \xi^2 (u^\epsilon-cx_2)\,dx - \int_U u^\ez_{ij} u^\ez_j \xi _i\xi (u^\epsilon-cx_2)^2 
   \,dx. \label{ineq}
 \end{align}
We apply the Cauchy-Schwartz inequality to the  third and forth terms in the right hand side of \eqref{ineq} 
\begin{align}
& -\int_{U}    u^\ez_{ij} u^\ez_j (u^\epsilon_i-c\delta_{2i})  \xi^2 (u^\epsilon-cx_2)\,dx - \int_U u^\ez_{ij} u^\ez_j \xi _i\xi (u^\epsilon-cx_2)^2 \,dx\nonumber\\
& \quad\le  \left[\int_U|D^2u^\epsilon Du^\epsilon|^2\xi^2\,dx\right]^{1/2} 
\left[
 \int_U(u^\epsilon-cx_2)^2(|c| +    |Du|)^2 \xi^2\,dx  + \int_U(u^\epsilon-cx_2)^4 |D\xi|^2 \,dx  \right]^{1/2}
\label{ineq 1}
\end{align} 
Keep in mind below that  $u^\epsilon_i  (u^\epsilon_i-c\delta_{2i}) = (|Du^\epsilon|^2-c u^\epsilon_2)$. 
Then via integration by parts, we  write the first  term  in the right hand side of \eqref{ineq} as  
\begin{align*}
 \int_{U} \Delta u^\ez u^\ez_i  (u^\epsilon_i-c\delta_{2i})  \xi^2 (u^\epsilon-cx_2)\,dx 
& =
-\int_U u^\ez_i[ (|Du^\epsilon|^2-c u^\epsilon_2)\xi^2 (u^\epsilon-cx_2) ]_i\,dx\\
& = -\int_U   (|Du^\epsilon|^2-c u^\epsilon_2)^2\xi^2  \,dx   \\
&\quad
-\int_U u^\ez_i(|Du^\epsilon|^2-c u^\epsilon_2)_i \xi^2 (u^\epsilon-cx_2) \,dx\\
&\quad  -2\int_U u^\ez_ i (|Du^\epsilon|^2-c u^\epsilon_2) \xi _i  \xi(u^\epsilon-cx_2) \,dx.
\end{align*}
By  the Cauchy-Schwartz inequality, 
\begin{align*}&-\int_U u^\ez_i (|Du^\epsilon|^2-c u^\epsilon_2)_i \xi^2 (u^\epsilon-cx_2) \,dx\\
&\quad= -\int_U 2 u^\ez_iu^\ez_{ij}u^\epsilon_j \xi^2 (u^\epsilon-cx_2)\,dx
+ c \int_U u^\epsilon_{2i}u^\epsilon_i \xi^2 (u^\epsilon-cx_2)\,dx \\
&\quad\le \left[\int_U|D^2u^\epsilon Du^\epsilon|^2\xi^2\,dx\right]^{1/2}  \left[ \int_U(u^\epsilon-cx_2)^2(|c| +    |Du^\ez|)^2 \xi^2\,dx \right]^{1/2}
\end{align*}
and 
\begin{align*}&-2 \int_U u^\ez_ i (|Du^\epsilon|^2-c u^\epsilon_2) \xi_i \xi(u^\epsilon-cx_2) \,dx
\\ &\quad
\le 2  \left[\int_U (u_i^\epsilon \xi_i)^2|Du^\epsilon|^2\,dx\right]^{1/2}
 \left[\int_U (u^\epsilon-cx_2)^2( |c|+|Du^\epsilon|)^2\xi^2\,dx\right]^{1/2}. 
\end{align*}

Therefore, using  Lemma \ref{cacciopoli ez}  to estimate $\int_U |D^2u^\epsilon Du^\epsilon|\xi^2\, dx $ we arrive at   
\begin{align}
&\int_{U} \Delta u^\ez u^\ez_i  (u^\epsilon_i-c\delta_{2i})  \xi^2 (u^\epsilon-cx_2)\le  -\int_U   (|Du^\epsilon|^2-c u^\epsilon_2)^2\xi^2  \,dx  \nonumber  \\
&\quad\quad+C\left[\int_{U} |Du^{\ez}|^4(|D\xi|^2 +|D^2\xi||\xi|)\,dx\right]^{1/2}\left[\int_U (u^\epsilon-cx_2)^2( |c|+|Du^\epsilon|)^2 \xi^2\,dx\right]^{1/2} . \label{ineq 2}
\end{align}

Finally, again via integration by parts, for the second  term  in the right hand side of \eqref{ineq} we have  
\begin{align}
\hspace{2.5cm} & \hspace{-2.5cm} \int_{U} \Delta u^\ez u^\ez_i \xi_i\xi (u^\epsilon-cx_2)^2   \,dx =-\int_{U} u^\epsilon_j  [ u^\ez_i \xi _i\xi (u^\epsilon-cx_2)^2]_j\,dx \nonumber \\
&=-\int_{U} (u_i^\epsilon \xi_i)^2  (u^\epsilon-cx_2)^2 \,dx - \int_{U} u^\epsilon_j   u^\ez_{ij} \xi _i\xi (u^\epsilon-cx_2)^2 \,dx\nonumber \\
&\quad  -\int_{U} u^\epsilon_j   u^\ez_i \xi _{ij}\xi (u^\epsilon-cx_2)^2 \,dx-2\int_{U} (|Du^\epsilon|^2-cu^\epsilon_2 )   u^\ez_i \xi _i\xi (u^\epsilon-cx_2) \,dx. 
\label{laplace} \end{align}
By the Cauchy-Schwartz  inequality,  we have 
\begin{align*}
- \int_{U} u^\epsilon_j   u^\ez_{ij} \xi _i\xi (u^\epsilon-cx_2)^2 \,dx\le   \left[\int_U|D^2u^\epsilon Du^\epsilon|^2\xi^2\,dx\right]^{1/2}  \left[ \int_U(u^\epsilon-cx_2)^4|D\xi|^2 \,dx \right]^{1/2},
\end{align*}
\begin{align*}
-\int_{U} u^\epsilon_j   u^\ez_i \xi _{ij}\xi (u^\epsilon-cx_2)^2 \,dx\le  \left[\int_U|  Du^\epsilon|^4|D^2\xi| |\xi|\,dx\right]^{1/2}  \left[ \int_U(u^\epsilon-cx_2)^4|D^2\xi| |\xi|\,dx \right]^{1/2},
\end{align*}
and 
\begin{align*}
&-\int_{U} (|Du^\epsilon|^2-cu^\epsilon_2 )   u^\ez_i \xi _i\xi (u^\epsilon-cx_2) \,dx\\
&\quad\le 
\frac14 \int_{U} (|Du^\epsilon|^2-cu^\epsilon_2 )^2    \xi^2 \,dx+  
4 \int_{U}   |Du^\ez|^2  |D\xi|^2 (u^\epsilon-cx_2) ^2\,dx\\
&\quad\le \frac14\int_{U} (|Du^\epsilon|^2-cu^\epsilon_2 )^2    \xi^2 \,dx+4   \left[\int_U| Du^\epsilon|^4|D \xi|^2 \,dx\right]^{1/2}  \left[ \int_U(u^\epsilon-cx_2)^4|D\xi|^2 \,dx\right]^{1/2}. 
\end{align*}
Thus (since the first term on the right hand side of \eqref{laplace} can easily be estimated) 
\begin{align}
&\int_{U} \Delta u^\ez u^\ez_i \xi_i\xi (u^\epsilon-cx_2)^2   \,dx \le \frac14 \int_U   (|Du^\epsilon|^2-c u^\epsilon_2)^2\xi^2  \,dx    \nonumber\\
&\quad\quad+C \left[\int_{U} |Du^{\ez}|^4(|D\xi|^2 +|D^2\xi||\xi|)\,dx\right]^{1/2}\left[\int_U (u^\epsilon-cx_2)^4( |D\xi|^2\xi^2+|D^2\xi||\xi|)\,dx \right]^{1/2}.  \label{ineq 3}
\end{align}

Combining \eqref{ineq} together with \eqref{ineq 1}, \eqref{ineq 2} and \eqref{ineq 3}, we complete   the proof of Lemma \ref{flat}. 
\end{proof}

\begin{lem}\label{better0}
Let $\alpha> 0$.   For any $\kappa>0$  and  $\xi \in C^{\infty}_c(U)$, we have
 \begin{align*} 
\hspace{2cm} & \hspace{-2cm}  \int_{U} [  ( |Du^{\ez}|  ^2+\kappa)|\xi|^2]^{  \alpha+1 }    \, dx  + \epsilon \alpha \int_U (|Du^{\ez}| ^2+\kappa)^{\alpha-2} ( \Delta u^{\ez})^2  |u^\ez|^2 \xi^2\,dx
\\ &  \le     C(\alpha )   \int_{U}     |u^\ez| ^{2 \alpha+2 }  (|D\xi|^2 +|D^2\xi||\xi|)^{\alpha+1}  \, dx\\
 &\quad+  \tilde C(\alpha) \epsilon   \int_U      (|Du^{\ez}| ^2+\kappa)^{\alpha-1} |u^\ez|^2    |D\xi|^2  \xi^{2\alpha}   \, dx\\
 &\quad+  (8\kappa+ \tilde C(\alpha)\epsilon)  \int_U   (|Du^{\ez}|^2+\kappa)^{\alpha}    |\xi|^{2(\alpha+1)}  \, dx.
\end{align*} 
\end{lem}

Lemma~\ref{sobolev of u} follows  from Lemma \ref{better0} by choosing a
suitable cut-off functions $\xi$.

\begin{proof}  [Proof of Lemma \ref{better0}]
We write the desired inequality as 
\begin{equation}  K_1  + \varepsilon \alpha K_2
 \le C(\alpha) J + C(\alpha) \varepsilon  E_1   + (8\kappa + \tilde C(\alpha) \epsilon) E_2.  
\end{equation} 
Let $\phi= (|Du^{\ez}| ^2+\kappa)^{  \alpha-1}|u^\ez|^2|\xi|^{2(\alpha+1)}$. 
Then $\phi\in W^{1,\,2}_c(U)$.
By \eqref{identityIII}, we write 
\begin{align*}  \mathbb I_\epsilon(\phi)&=\int_{U}  |D|Du^\epsilon|(z)|^2 (|Du^{\ez}| ^2+\kappa)^{  \alpha-1 }|u^\ez|^2\xi^{2(\alpha+1)} \, dx\\
&\quad   +\ez\int_{U}  \frac{(\Delta u^\ez)^2}{|Du^\ez|^2}(|Du^{\ez}| ^2+\kappa)^{  \alpha-1 }|u^\ez|^2\xi^{2(\alpha+1)}\, dx=:J_1+J_2.
\end{align*}  
We compute the derivative of $\phi$, 
\begin{align*}
\phi_i &=2(\alpha-1) (|Du^{\ez}| ^2+\kappa)^{\alpha-2 } u^{\ez}_{ij}u^{\ez}_j |u^\ez|^2\xi^{2(\alpha+1)}\\
&\quad+ 2(\alpha+1)  |\xi|^{2\alpha}  \xi \xi_i (|Du^{\ez}| ^2+\kappa)^{\alpha-1 } |u^\ez|^2+
 2 (|Du^{\ez}| ^2+\kappa)^{ \alpha-1 }u^\epsilon_i u^\epsilon \xi^{2(\alpha+1)}. 
 \end{align*}
As above we have by  \eqref{functional}, 
\begin{align*}
\mathbb I_\epsilon(\phi)& = -(\alpha-1)\int_{U}  (|Du^{\ez}| ^2+\kappa)^{\alpha-2}|D^2 u^{\ez} Du^{\ez}|^2 |u^\ez|^2|\xi|^{2(\alpha+1)} \,dx\\
&\quad 
 - (\alpha+1) \int_U   (|Du^{\ez}| ^2+\kappa)^{\alpha-1} u^{\ez}_{ij}u^{\ez}_j|\xi|^{2\alpha} \xi_i \xi |u^\ez|^2\,dx\\
&\quad 
 - \int_U  (|Du^{\ez}| ^2+\kappa)^{\alpha-1} u^{\ez}_{ij}u^{\ez}_j   u^\epsilon_i u^\ez  |\xi|^{2(\alpha+1)}  \,dx\\
&\quad +(\alpha-1)\int_U (|Du^{\ez}| ^2+\kappa)^{\alpha-2} \Delta u^{\ez} u^{\ez}_{ij}u^{\ez}_i u^{\ez}_j |u^\ez|^2\xi^{2(\alpha+1)}\,dx\\
 &\quad +(\alpha+1) \int_U     (|Du^{\ez}| ^2+\kappa)^{\alpha-1} \Delta u^{\ez} u^{\ez}_i\xi_i  |\xi|^{2\alpha}  |u^\ez|^2\, dx\\
&\quad+ \int_U   (|Du^{\ez}| ^2+\kappa)^{\alpha-1} \Delta u^{\ez} u^{\ez}_i   u^\epsilon _i u^\ez |\xi|^{2(\alpha+1)}\, dx\\
&= I_1+\cdots+I_6. 
\end{align*}
Notice that
$$ J_1-I_1\ge \alpha \int_{U}  (|Du^{\ez}| ^2+\kappa)^{\alpha-2}|D^2 u^{\ez} Du^{\ez}|^2 |u^\ez|^2|\xi|^{2(\alpha+1)}\,dx\ge 0.$$
Since  $u^{\ez}_{ij}u^{\ez}_ju^\ez_i=\Delta_\infty u^\ez=-\epsilon\Delta u^\ez$, we write $I_4$ as  
\begin{align*}I_4&= -\epsilon(\alpha-1) \int_U (|Du^{\ez}| ^2+\kappa)^{\alpha-2} ( \Delta u^{\ez})^2  |u^\ez|^2|\xi|^{2(\alpha+1)}\,dx.
\end{align*}
Hence, 
$$J_2-I_4\ge   \ez\alpha   K_2.$$
To complete the proof it suffices to show that 
\begin{equation}\label{i2}  I_2\le \frac18   K_1 + C(\alpha )  J \end{equation}  
\begin{equation}\label{i3}  I_3\le \frac18 \varepsilon \alpha K_2 + C(\alpha) \varepsilon E_1   \end{equation}  
\begin{equation}\label{i5}  I_5 \le \frac14    K_1 + \frac14 \varepsilon \alpha  K_2+C(\alpha)J + C(\alpha) \varepsilon E_1   \end{equation} 
and  
\begin{equation}\label{i6} I_6 \le -  \frac78 K_1 + \frac18 \varepsilon \alpha  K_2  + C(\alpha)J   + (C(\alpha) \varepsilon+ 4 \kappa)  E_2. \end{equation}  

Below we prove  \eqref{i2} to \eqref{i6} in order.  Recall that 
by H\"older's and   and Young's inequality we have for $\frac1p+\frac1q = 1$,
$1<p,q< \infty$ and $\eta>0$ 
\begin{eqnarray}  \int fg dx\le &&  
\Vert f \Vert_{L^p} \Vert g \Vert_{L^q}  
\le   \eta \Vert f \Vert_{L^p}^p +  C(\eta,p) \Vert g \Vert_{L^q}^q,\label{young}
\end{eqnarray}   
which will be used later for different choice of $p$ (and hence $q$)  and $\eta $. 

Write 
\begin{align}I_2&= - (\alpha+1) \int_U   (|Du^{\ez}| ^2+\kappa)^{\alpha-1} u^{\ez}_{ij}u^{\ez}_j|\xi|^{2\alpha} \xi_i \xi |u^\ez|^2\,dx\nonumber
\\ 
&=-\frac{\alpha+1}{2\alpha }\int_{U}  [ (|Du^{\ez}| ^2+\kappa)^{\alpha }]_i \xi_i \xi^{2\alpha-1}  |u^\ez|^2\, dx\nonumber\\
&= \frac{\alpha+1}{2\alpha }\int_{U}  (|Du^{\ez}| ^2+\kappa)^{\alpha } [\xi_i  
\xi^{2\alpha+1}  |u^\ez|^2]_i\, dx\nonumber\\
&\le  \frac{\alpha+1}{2\alpha  }\int_{U} [|\xi|^2  (|Du^{\ez}| ^2+\kappa)  ]^{\alpha }[\Delta \xi \xi +(2\alpha+1) |D\xi|^2 ] |u^\ez|^2  \, dx\nonumber\\
&\quad+ \frac{ (\alpha+1)}{ \alpha  }  \int_{U} 
[|\xi|^2  (|Du^{\ez}|^2+\kappa)  ]^{\alpha+1/2 }
   |D\xi|    |u^\ez|    \, dx.\label{i2x}
\end{align}

Applying \eqref{young} with $\eta = \frac{1}{16 } $ and $p= \frac{\alpha+1}{\alpha}$ in the first term of \eqref{i2x} or 
 $p = \frac{2\alpha+2}{2\alpha+1}$ in the second term  of \eqref{i2x},  we obtain \eqref{i2}, that is, 
 \begin{align*}I_2 
 & \le   \frac1{8} K_1   +  C(\alpha)      J .
 \end{align*}

Replacing $u^{\ez}_{ij}u^{\ez}_ju^\ez_i$ by $-\epsilon\Delta u^\ez$ in $I_3$ and using \eqref{young} with $p=2$ and $\eta=\frac{\alpha}8$, we have 
\begin{align*} I_3&=
 \epsilon \int_U (|Du^{\ez}| ^2+\kappa)^{\alpha-1}  \Delta u^\ez u^\ez   \xi^ {2(\alpha+1) }\,dx\\
& \le \frac18 \varepsilon \alpha K_2  + C(\az)\epsilon  \int_U (|Du^{\ez}| ^2+\kappa)^{\alpha }       \xi^ {2(\alpha+1) }\,dx 
\\ & = \frac18 \varepsilon \alpha K_2  + C(\az)\epsilon  E_1,
\end{align*}
which gives  \eqref{i3}. 

By integration by parts, write 
\begin{align*} I_5
&= - (\alpha+1)   \int_U    u^{\ez}_j  [(|Du^{\ez}| ^2+\kappa)^{\alpha-1} u^{\ez}_i\xi_i   \xi^ {2\alpha+1}  |u^\ez|^2]_j\, dx\\
&= - { 2(\alpha+1) }(\alpha-1)   \int_U     (|Du^{\ez}| ^2+\kappa)^{\alpha-2} u^{\ez}_j  u^\ez_{js} u^{\ez}_s   u^{\ez}_i \xi_{i} \xi^ {2\alpha+1}  |u^\ez|^2  \, dx\\
&\quad-(\alpha+1)     \int_U    u^{\ez}_j  u^{\ez}_{ij}\xi_i  (|Du^{\ez}| ^2+\kappa)^{\alpha-1}    \xi^ {2\alpha+1}  |u^\ez|^2\, dx\\
&\quad-(\alpha+1)  \int_U    u^{\ez}_j  u^{\ez}_i\xi_{ij}  (|Du^{\ez}| ^2+\kappa)^{\alpha-1}  \xi^ {2\alpha+1}  |u^\ez|^2 \, dx\\
&\quad-{ (\alpha+1)  (2\alpha+1)}  \int_U    u^{\ez}_j  u^{\ez}_i\xi_{i}\xi_j  (|Du^{\ez}| ^2+\kappa)^{\alpha-1}  \xi^ {2\alpha }  |u^\ez|^2 \, dx\\
&\quad-{ 2(\alpha+1)  }\int_U    u^{\ez}_j u^{\ez}_j   (|Du^{\ez}| ^2+\kappa)^{\alpha-1}  u^{\ez}_i \xi_{i} \xi^ {2\alpha+1} u^\epsilon   \, dx\\
&=I_{5,1}+\cdots+I_{5,5}
\end{align*}
Replacing $u^{\ez}_{ js}u^{\ez}_ju^\ez_s=\Delta_\infty u^\ez$ by $-\epsilon\Delta u^\ez$ in $I_{5,1}$ and using  using \eqref{young} with $p=2$ and $\eta=\frac{\alpha}8$, we have 
\begin{align*}I_{5,1}&=  { 2(\alpha+1) }(\alpha-1) \epsilon   \int_U     (|Du^{\ez}| ^2+\kappa)^{\alpha-2}   \Delta u^\ez    u^{\ez}_i \xi_{i} \xi^ {2\alpha+1}  |u^\ez|^2  \, dx\\
&\le \frac18 \varepsilon \alpha   K_2+ C( \alpha)\epsilon   \int_U     (|Du^{\ez}| ^2+\kappa)^{\alpha-1}     | D \xi|^2 \xi^ {2\alpha } |u^\ez|^2  \, dx
\\ &  = \frac18 \varepsilon \alpha K_2 + C(\alpha) \epsilon E_1. 
\end{align*}
Note that $I_{5,2}=I_2$ and hence have the same estimate as $I_2$ above. 
Applying \eqref{young} with $p=\frac{\alpha+1}{\alpha}$ and $\eta=\frac{\gamma(2\gamma-1)}{16}$  we have 
 \begin{align*}
I_{5,3}& =- (\alpha+1)   \int_U    u^{\ez}_j  u^{\ez}_i\xi_{ij}  (|Du^{\ez}| ^2+\kappa)^{\alpha-1}  \xi^ {2\alpha+1} |u^\ez|^2 \, dx
 \le   \frac1{16}  K_1   +   C (\alpha )  J.
\end{align*}

Moreover, 
 \begin{align*}
I_{5,4}&=
 - { (\alpha+1)(2\alpha+1)}   \int_U     \langle Du^{\ez}, D\xi\rangle ^2   (|Du^{\ez}| ^2+\kappa)^{\alpha-1}  \xi^ {2\alpha } |u^\ez|^2 \, dx\le 0. 
\end{align*}
By  \eqref{young} with $p=\frac{2\alpha+2}{2\alpha+1}$ and $\eta=\frac1{8}$ we also have 
 \begin{align*}
I_{5,5}
&\le   2(\alpha+1)    \int_U    (|Du^{\ez}| ^2+\kappa)^{\alpha +1/2}   |D\xi| \xi^ {2\alpha+1}   (|u^\ez|^2+\tau )^{\gamma-1/2 } \, dx\le \frac1{16}  K_1  +   C (\alpha )J . 
\end{align*}
Combining the estimates for $I_{5,1}$ to $I_{5,5}$, we conclude   \eqref{i5}. 

By integration by parts, we have 
 \begin{align*}
I_6
&=-   \int_U u^\epsilon_j [ (|Du^{\ez}| ^2+\kappa)^{\alpha-1}  u^{\ez}_i   u^\epsilon _i u^\ez  \xi^ {2(\alpha+1) }]_j\, dx\\
&= -  2  \int_U u^\epsilon_j  u^\ez_{js}u^\ez_s  (|Du^{\ez}| ^2+\kappa)^{\alpha-2}  u^{\ez}_i   u^\epsilon _i u^\ez \xi^ {2(\alpha+1) } \, dx\\
&\quad- 2 \int_U u^\epsilon_j   (|Du^{\ez}| ^2+\kappa)^{\alpha-1}  u^{\ez}_{ij }  u^\epsilon _i u^\ez \xi^ {2(\alpha+1) } \, dx\\
&\quad- \int_U u^\epsilon_j   (|Du^{\ez}| ^2+\kappa)^{\alpha-1}  u^{\ez}_{i  }  u^\epsilon _i u^\ez_j  \xi^ {2(\alpha+1) } \, dx\\
&\quad- {2(\alpha+1)} \int_U u^\epsilon_j   (|Du^{\ez}| ^2+\kappa)^{\alpha-1}  u^{\ez}_{i  }  u^\epsilon _i u^\ez \xi _i\xi^ {2\alpha+1 } \, dx\\
&=I_{6,1}+\cdots+I_{6,4}.
\end{align*} 
Replacing $u^{\ez}_{ js}u^{\ez}_ju^\ez_s=\Delta_\infty u^\ez$ by $-\epsilon\Delta u^\ez$ in $I_{6,1},I_{6,2}$ and using Young's inequality \eqref{young} with $p=2$ and $\eta = \alpha/16$, we have 
 \begin{align*}
I_{6,1}+I_{6,2}
&=  2 \epsilon \int_U  \Delta u^\ez   (|Du^{\ez}| ^2+\kappa)^{\alpha-2}   |Du^{\ez}|^2   u^\ez  \xi^ {2(\alpha+1) } \, dx\\
&\quad+
  2 \ez  \int_U \Delta u^\epsilon    (|Du^{\ez}| ^2+\kappa)^{\alpha-1}    u^\ez  \xi^ {2(\alpha+1) } \, dx\\
& \le  \frac18 \varepsilon \alpha K_2  +  C(\alpha )\epsilon   \int_U      (|Du^{\ez}| ^2+\kappa)^{\alpha }    \xi^ {2(\alpha+1) } \, dx
\\ & = \frac18 \varepsilon \alpha K_2 + C(\alpha) \varepsilon E_2. 
\end{align*} 
Write 
 \begin{align*} I_{6,3} & =  -   \int_U   (|Du^{\ez}| ^2+\kappa)^{\alpha-1}  | D u^{\ez}|^4 \xi^ {2(\alpha+1) } \, dx \\
 &\le  -  K_1 +2\kappa \int_U   (|Du^{\ez}| ^2+\kappa)^{\alpha} \xi^ {2(\alpha+1) } \, dx
\\&  = - K_1 +  4\kappa E_2. 
\end{align*} 
By \eqref{young} with $p=\frac{2\alpha+2}{2\alpha+1}$ and $\eta=\frac18$    we have 
 \begin{align*} I_{6,4}
&\le    2(\alpha+1)   \int_U    (|Du^{\ez}| ^2+\kappa)^{\alpha+1/2}  | D\xi| u^\ez  \xi^ {2\alpha+1 } \, dx\le  \frac18 K_1+ C (\alpha )J.
\end{align*} 
Combining the estimates for $I_{6,1}$ to $I_{6,4}$, we conclude   \eqref{i6}. This complete the proof of Lemma \ref{better0}. 
\end{proof}

\bigskip 
 \noindent {\bf Acknowledgment.} H. Koch and Y. Zhang have been
 partially supported by the Hausdorff Center for Mathematics.  Y. Zhou
 would like to thank the supports of von Humboldt Foundation, and
 National Natural Science of Foundation of China (No. 11522102).

\end{document}